\documentclass[a4paper,english]{article}
\pdfoutput=1

\usepackage{xcolor}
\definecolor{lightgray}{gray}{0.75}
\newcommand\greybox[1]{%
  \vskip\baselineskip%
  \par\noindent\colorbox{lightgray}{%
    \begin{minipage}{\textwidth}#1\end{minipage}%
  }%
  \vskip\baselineskip%
}

\usepackage[font=small]{caption}

\providecommand{\norm}[1]{\lVert#1\rVert}
\usepackage{hyperref}
\usepackage[color,notref,notcite]{showkeys}
\definecolor{refkey}{RGB}{255,0,0}
\definecolor{labelkey}{RGB}{255,0,0}

\usepackage[OT1]{fontenc}
\usepackage[macce]{inputenc}
\usepackage{babel}
\usepackage{graphicx}
\usepackage{epstopdf}
\usepackage{amsmath}
\usepackage{amsfonts}
\usepackage{amssymb}
\usepackage{amsthm}
\usepackage{multirow}

\newcommand{\inputdir}{.}
\graphicspath{{.}}

\usepackage{tikz}
\usetikzlibrary{decorations.markings,arrows,patterns,positioning,plotmarks}
\pgfdeclareimage[interpolate=true,width=6.5cm]{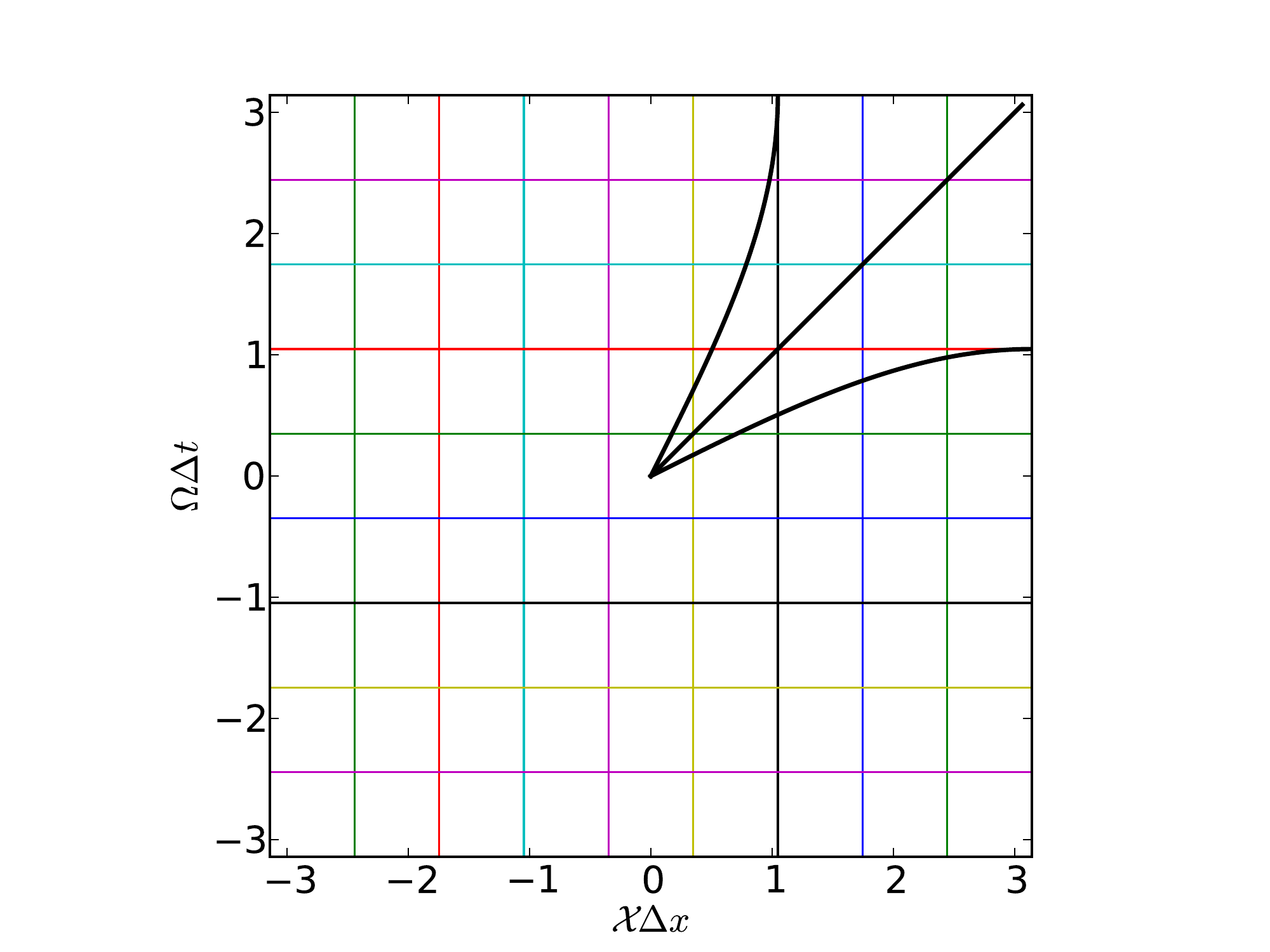}{dispersion_linear_domain}
\pgfdeclareimage[interpolate=true,width=6.5cm]{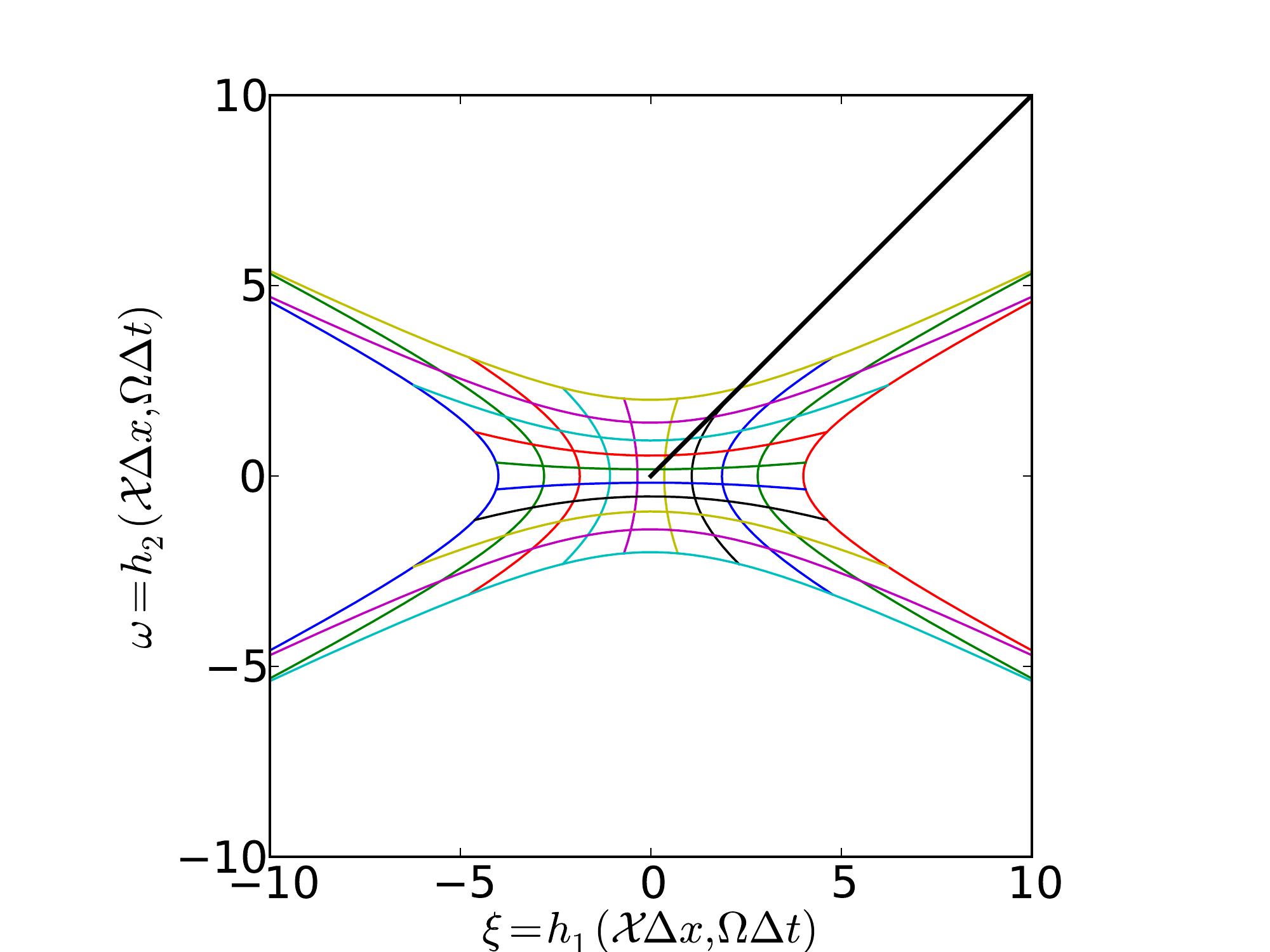}{dispersion_linear_range_0}
\pgfdeclareimage[interpolate=true,width=6.5cm]{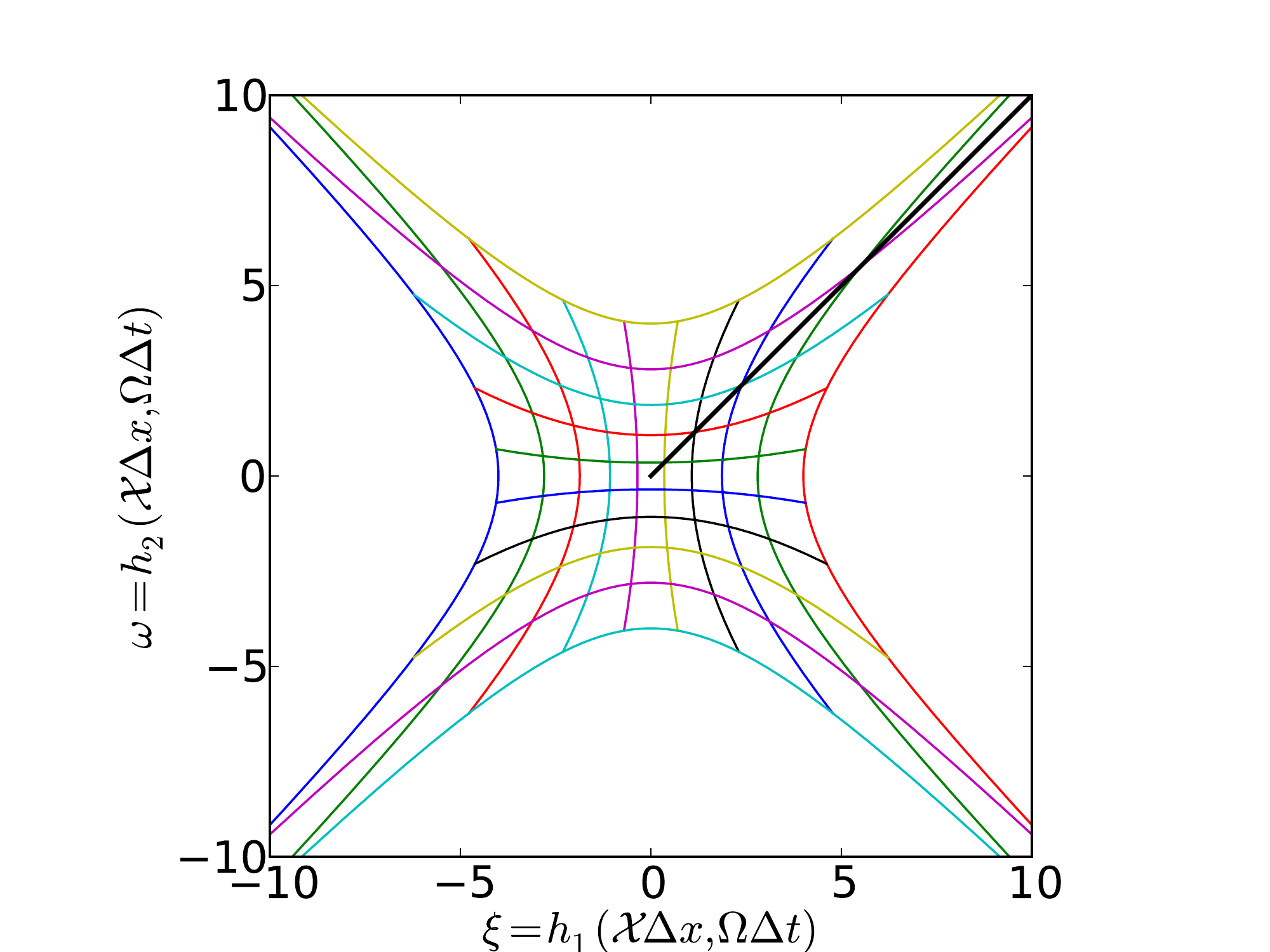}{dispersion_linear_range_1}
\pgfdeclareimage[interpolate=true,width=6.5cm]{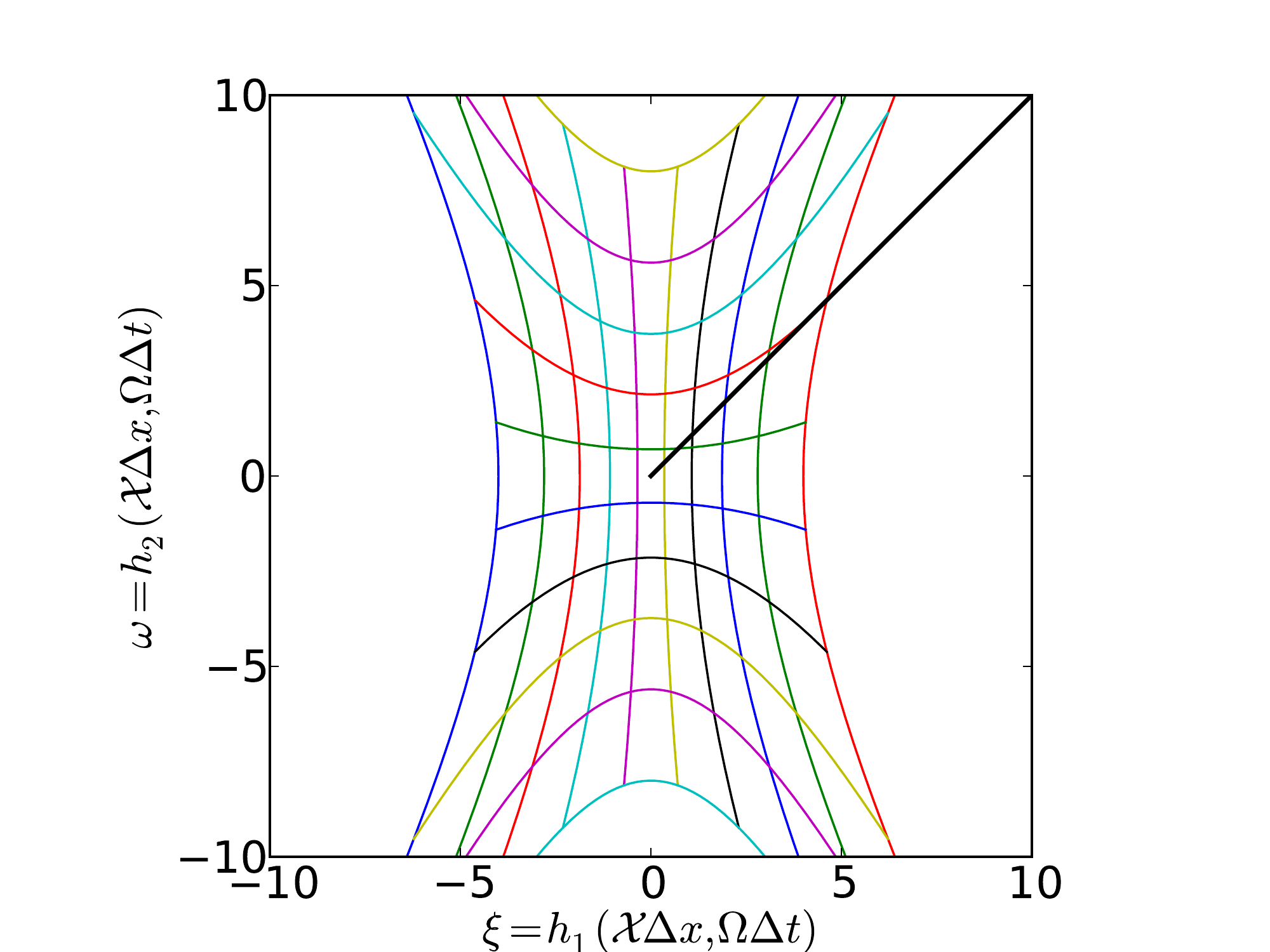}{dispersion_linear_range_2}
\pgfdeclareimage[interpolate=true,width=7.5cm]{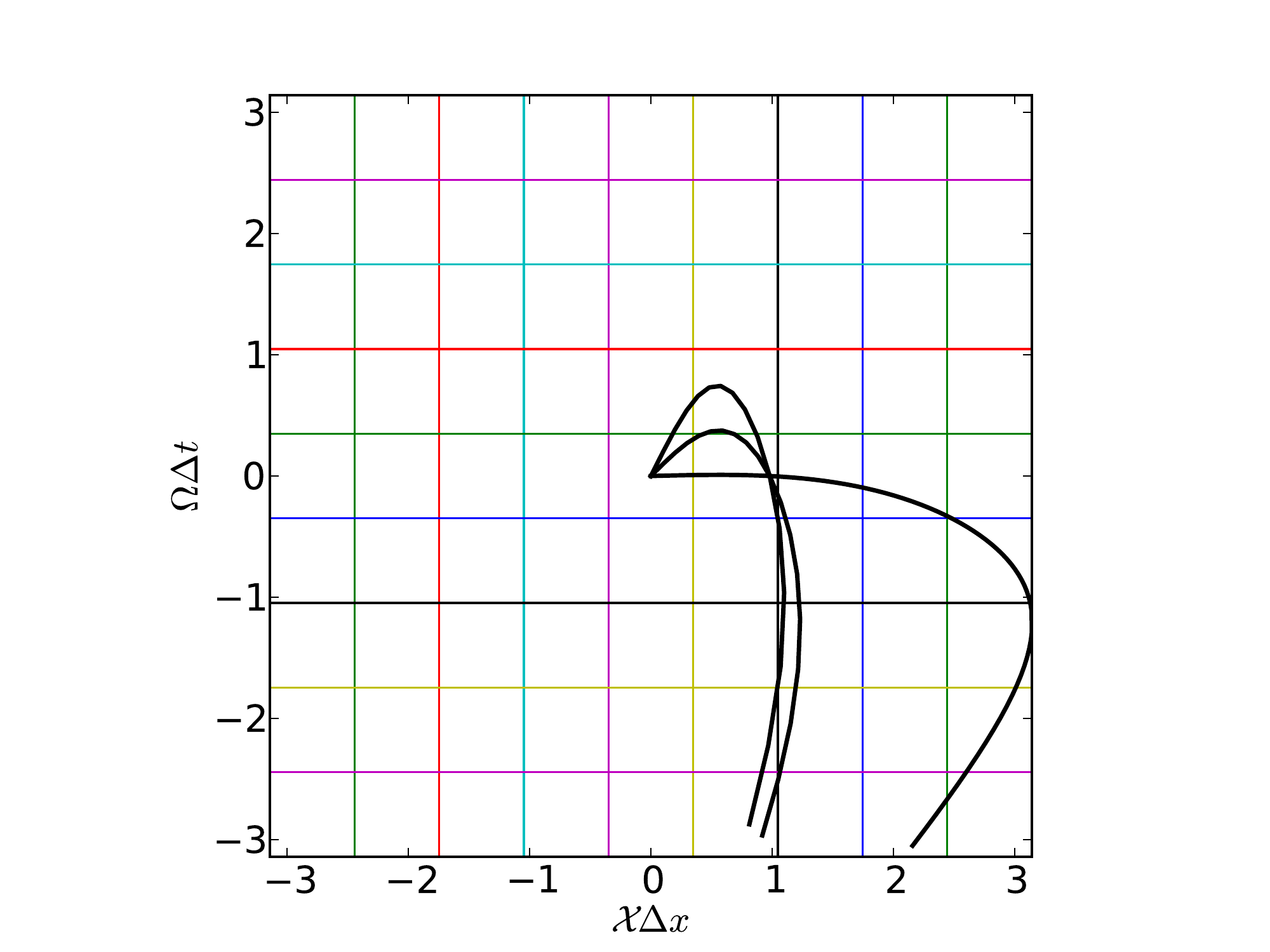}{dispersion_cubic_domain}
\pgfdeclareimage[interpolate=true,width=7.5cm]{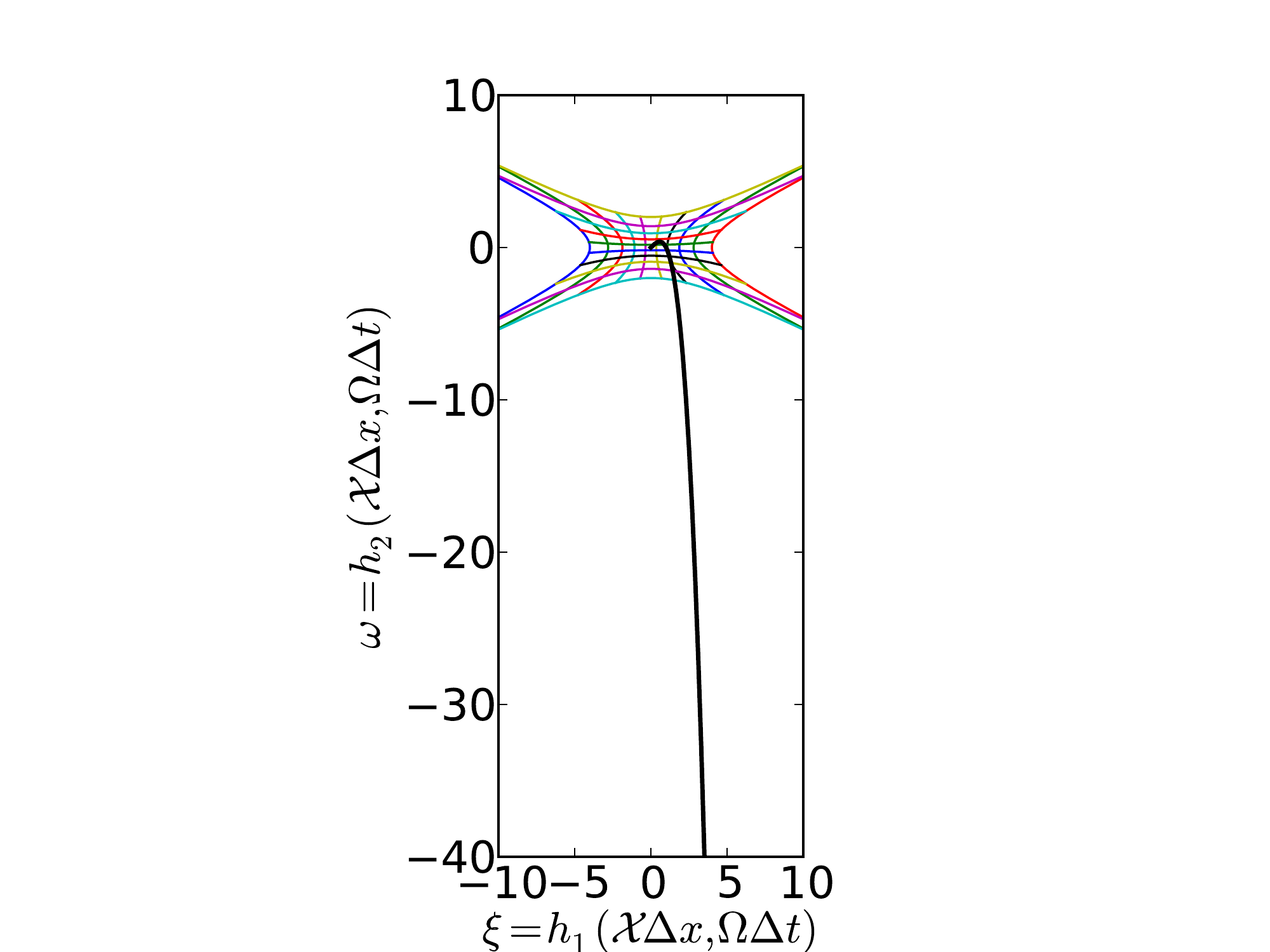}{dispersion_cubic_range_0}
\pgfdeclareimage[interpolate=true,width=7.5cm]{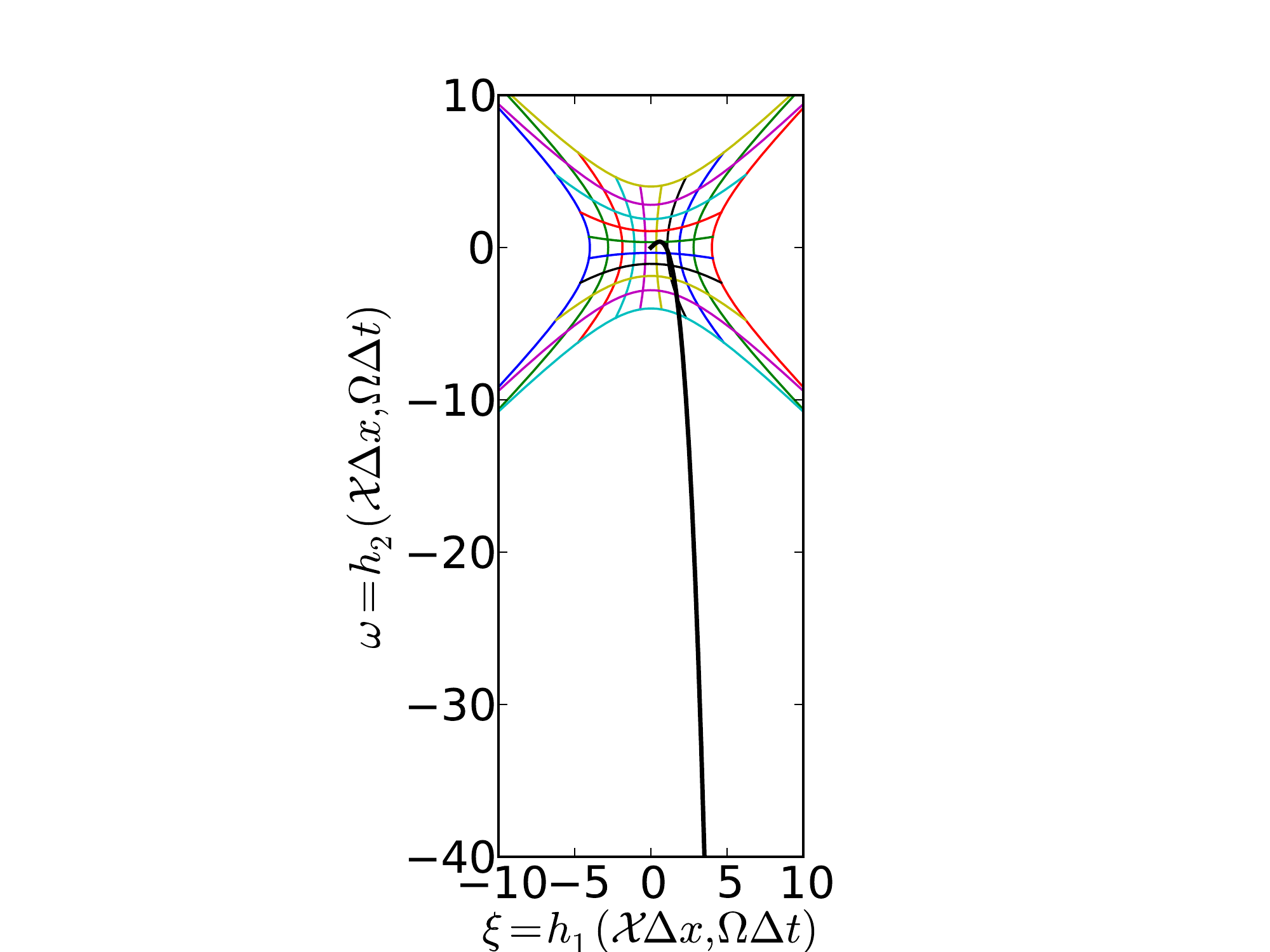}{dispersion_cubic_range_1}
\pgfdeclareimage[interpolate=true,width=7.5cm]{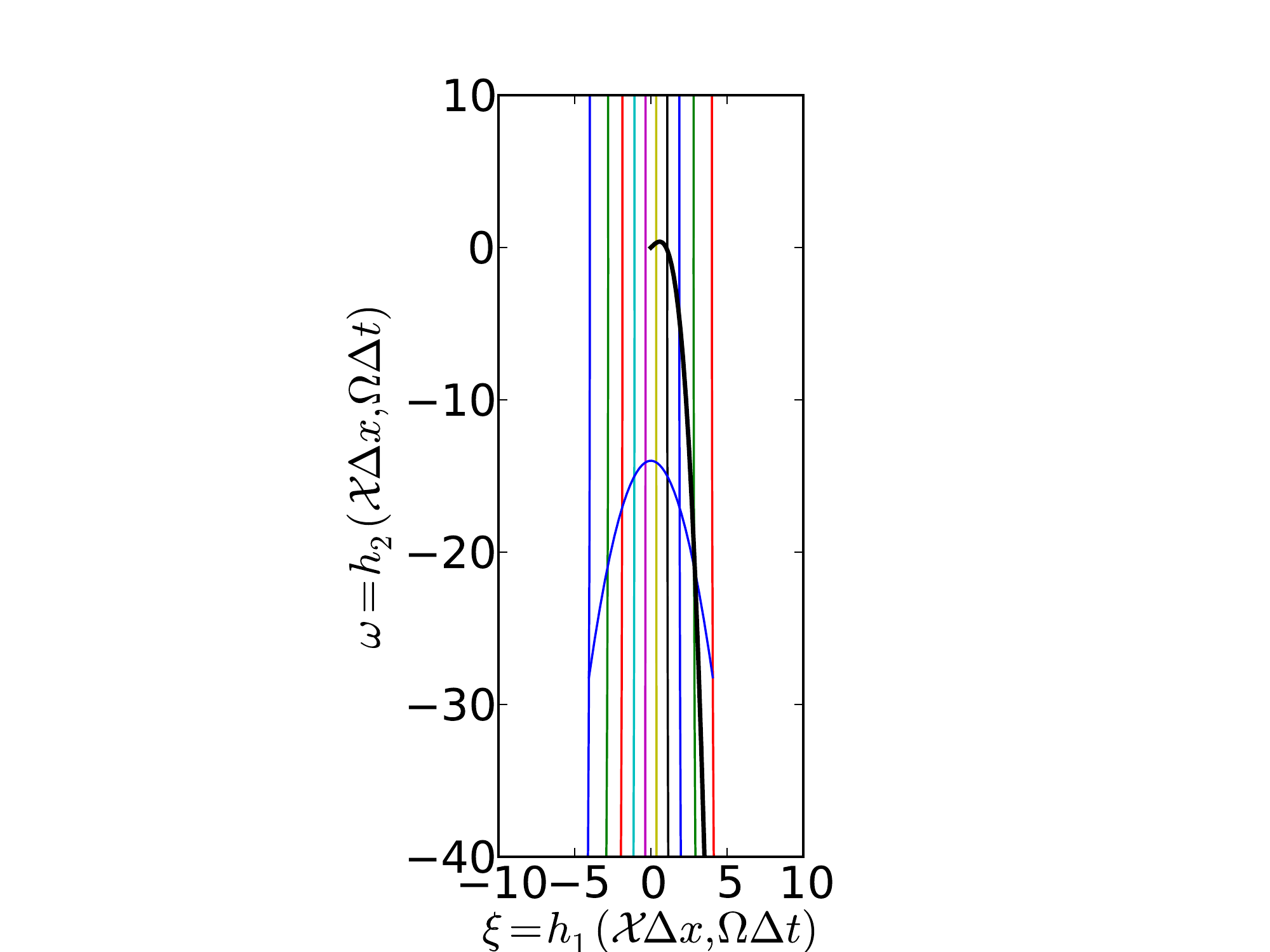}{dispersion_cubic_range_2}

\newtheorem{thm}{Theorem}
\newtheorem{cor}[thm]{Corollary}
\newtheorem{prop}[thm]{Proposition}
\newtheorem{lem}[thm]{Lemma}
\newtheorem{dfn}[thm]{Definition}

\newcommand{\BigO}[1]{\mathcal{O}\left(#1\right)}

\pgfdeclareplotmark{triangle*}
{%
  \pgfpathmoveto{\pgfqpoint{0pt}{\pgfplotmarksize}}%
  \pgfpathlineto{\pgfqpointpolar{-30}{\pgfplotmarksize}}%
  \pgfpathlineto{\pgfqpointpolar{-150}{\pgfplotmarksize}}%
  \pgfpathclose
  \pgfusepathqfillstroke
}

\begin{document}

\title{The multisymplectic diamond scheme}
\author{R I McLachlan%
\thanks{Institute of Fundamental Sciences, Massey University,
   Private Bag 11 222, Palmerston North 4442, New Zealand}
   \and M C Wilkins$^*$}

\maketitle

\begin{abstract}\noindent
We introduce a class of general purpose linear multisymplectic integrators for Hamiltonian wave equations based on a diamond-shaped mesh. On each diamond, the PDE is discretized by a symplectic Runge--Kutta method. The scheme advances in time by filling in each diamond locally, leading to greater efficiency and parallelization and easier treatment of boundary conditions compared to methods based on rectangular meshes. 
\end{abstract}

\pagestyle{myheadings}
\thispagestyle{plain}
\markboth{R I MCLACHLAN AND M C WILKINS}{The multisymplectic diamond scheme}

\section{Introduction}
In this paper we consider multisymplectic integrators for the
the multi-Hamiltonian PDE
\begin{equation}
   K\mathbf{z}_t + L\mathbf{z}_x = \nabla S(\mathbf{z}), \label{eqn:hampde}
\end{equation}
where $K$ and $L$ are constant $n \times n$ real skew-symmetric matrices,
$\mathbf{z}\colon\Omega\to \mathbb{R}^n$, $\Omega\subset\mathbb{R}^2$, and $S\colon \mathbb{R}^n
\to\mathbb{R}$.  
By introducing $\mathbf{z} = (u, v, w)$, $v = u_t$ and $w = u_x$, the one-dimensional wave
equation, $u_{tt} - u_{xx} = f(u)$, can be written in this
form with
\begin{gather} \label{eqn:1dwaveeqn}
   K = \begin{pmatrix} 0 & -1 & 0\\1 & 0 & 0\\0 & 0 & 0 \end{pmatrix},
   \quad L = \begin{pmatrix} 0 & 0 & 1\\0 & 0 & 0\\-1 & 0 & 0
   \end{pmatrix}, \\ \quad S(\mathbf{z}) = -V(u)+\frac{1}{2}v^2 - \frac{1}{2}w^2, \;\mathrm{and}\;
   f(u)=-V'(u).
\end{gather} 
Many variational PDEs can be written in the canonical form~\eqref{eqn:hampde}, including the
Schr\"{o}dinger, Korteweg--de Vries, and Maxwell equations.
Any  solutions to~\eqref{eqn:hampde} satisfy the \emph{multisymplectic conservation law}
\begin{equation}
   \omega_t + \kappa_x = 0,    \label{eqn:symlaw}
\end{equation}
where $\omega = \frac{1}{2} (d\mathbf{z} \wedge K d\mathbf{z})$ and
$\kappa = \frac{1}{2} (d\mathbf{z} \wedge L d\mathbf{z})$~\cite[pg.~338]{simhamdyn}.
A numerical method that satisfies a discrete version of
Eq.~\eqref{eqn:symlaw} is called a \emph{multisymplectic integrator};
see \cite{simhamdyn,bridges2006numerical} for reviews of multisymplectic integration.

The (Preissman or Keller) box scheme \cite[p.~342]{simhamdyn}, a multisymplectic
integrator, is simply the implicit midpoint rule (a Runge--Kutta method)
applied in space and in time on a rectangular grid.
We call it the {\em simple box scheme} to distinguish it from 
other Runge--Kutta box-based schemes.
There are plenty of multisymplectic low-order methods applicable to
Schr{\"o}dinger's equation~\cite{ zhong2013new, liao2013nonstandard,
islas2001geometric, sun2003multi, hong2006globally,
chen2002symplectic}.  Most are based on box-like schemes and
are second order.  LingHua's~\cite{LingHua2013} method for the Klein-Gordon-Schr{\"o}dinger equation
has spectral
accuracy in space and is second order in time.
Jia-Xiang~\cite{jia2013multisymplectic} presents a multisymplectic,
low order, implicit/explicit method  for the {K}lein-{G}ordon-{S}chr{\"o}dinger equation.
Hong~\cite{hong2006multi} presents a box-like multisymplectic method for the 
nonlinear Dirac equation.  For the Korteweg-de Vries equation there
are numerous~\cite{ hong2006multi, wang2012multi,
zhao2000multisymplectic, ascher2005symplectic} multisymplectic low
order  box-like schemes. Moore~\cite{moore2003multi} gives a multisymplectic
box-like low order scheme that can be applied to any multi-Hamiltonian
system.
Bridges and Reich~\cite{Bridges2001184} present a staggered-grid
multisymplectic method that is based on the symplectic
St{\"o}rmer-Verlet scheme.  They discuss possible extensions to higher
order methods.
 
Instead of discretizing one particular PDE, we wish to develop methods
that are applicable to the entire class~\eqref{eqn:hampde}, specializing to a particular equation or family as late as possible. The simple box scheme is
simple to define, can in principal be applied to any PDE of the
form~\eqref{eqn:hampde} and has several appealing properties, including the unconditional preservation of dispersion relations (up to diffeomorphic remapping of continuous to 
discrete frequencies) with consequent lack of parasitic waves \cite{ascher2004multisymplectic} and preservation of the sign of group velocities \cite{frank2006linear}, and lack of spurious reflections at points where the mesh size changes \cite{frank2004on}. These properties are related
to the linearity of the box scheme which suggests that this feature should be retained. 

However, the simple box scheme also has some less positive features. It is fully implicit, which makes it expensive; for equations where the CFL condition is not too restrictive, the extra (linear and sometimes nonlinear) stability this provides is not needed. The implicit equations may not have a solution: with periodic boundary conditions, solvability requires that the number of grid points be odd \cite{ryland}; we have found no general treatment of Dirichlet, Neumann, or mixed boundary conditions in the literature that leads to a well-posed method. It is only second order in space and time.

The latter issue can be avoided by applying higher-order Runge--Kutta methods in space and in time \cite{reich}. As the dependent variables are the internal stages, typically one obtains the stage order in space, for example order $r$ for $r$-stage Gauss Runge--Kutta \cite{mclachlan2014on}. However the scheme is still fully implicit and this time leads to singular ODEs for periodic boundary conditions unless $r$ and $N$ are {\em both} odd \cite{ryland,mclachlan2014on}.

The first two issues, implicitness and boundary treatment, are related. They can be avoided for some PDEs, like the nonlinear wave equation, by applying suitably partitioned Runge--Kutta methods \cite{ryland, mclachlan2011linear, ryland2008multisymplecticity, ryland2007multisymplecticity}. When they apply, they lead to explicit ODEs amenable to explicit time-stepping, can have high order, and can deal with general boundary conditions. However, the partitioning means that they are not linear methods.

The analogy with Hamiltonian ODEs, for which explicit and implicit symplectic integrators both have their domain of applicability, is striking, and indicates that there may be multisymplectic integrators based on Runge--Kutta discretizations that respect the structure of the PDE better and lead to broader applicability. This is the case, and we introduce in this paper the class of {\em diamond schemes} for~\eqref{eqn:hampde}. It is based on the following observation. Let the 
PDE~\eqref{eqn:hampde} be discretized on a square cell by a Runge--Kutta method in space and time. To each internal point there are $n$ equations and $n$ unknowns. To each pair of opposite edge points there is one equation. Therefore, to get a closed system with the same number of equations as unknowns, data should be specified on exactly {\em half} of the edge points. We shall show later that for the nonlinear wave equation, specifying $z$ at the edge points on two adjacent edges leads to a properly determined system for the two opposite edges. What remains is to arrange the cells so that the information flow is consistent with the initial value problem.

\begin{dfn}\label{def:diamond}
 A {\em diamond scheme} for the PDE~\eqref{eqn:hampde} is a quadrilateral mesh in space-time together with a mapping of each quadrilateral  to a square to which a Runge--Kutta method is applied in each dimension, together with
initial data specified at sufficient edge points such that the solution can be propagated forward in time by locally solving for pairs of adjacent edges. 
\end{dfn}

For equations that are symmetric in $x$, the quadrilaterals are typically diamonds, and we outline the scheme first in the simplest case, the analogue of the simple box scheme that we call the {\em simple diamond scheme}.

The diamond scheme is inspired by and has some similarities with the {\em staircase method} in discrete integrability \cite{van2010staircase}. In both cases initial data is posed on a subset of a quadrilateral graph such that the remaining data can be filled in uniquely. In discrete integrability, this fill-in is usually explicit, whereas for the diamond schemes  it depends on the PDE and is usually implicit. A second key point is that in the diamond method, there is a stability condition that the fill-in must be such that the numerical domain of dependence includes the analytic domain of dependence. Thus the characteristics of the PDE determine the geometry of the mesh: if they all pointed to the right then one could indeed use a simple rectangular mesh and fill in from left to right. (Indeed, this was how early versions of the box scheme proceeded.)

\section{The simple diamond scheme}

Consider solving Eq.~\eqref{eqn:hampde} numerically on the domain $x \in
\left[a,b\right]$, $t \ge 0$, with periodic boundary conditions.  Unlike a typical finite difference scheme
which uses a rectilinear grid aligned with the $(x,t)$ axes, the
simple diamond scheme uses
a mesh comprised of diamonds, see
figure~\eqref{fig:simple_diamond_mesh}.
   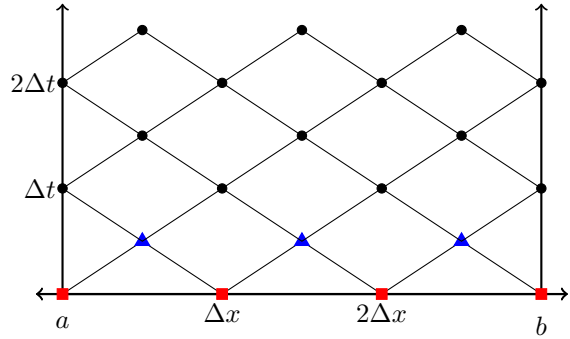
\begin{figure}
      \centering
      \begin{tikzpicture}[scale=0.35]

         \draw[<->,thick] (0, 11) -- (0, 0) -- (18, 0) -- (18, 11);
         \draw[<->,thick] (-1, 0) -- (19, 0);
         \draw[<->,thick] (-1, 0) -- (19, 0);
         
         \foreach \x in {0cm, 6cm, 12cm} {
            \draw [xshift=\x] (0, 0) -- (3, 2) -- (6, 0);
         }
         \node[color=red] at (0,0) {\pgfuseplotmark{square*}};    
         \node[color=red] at (6,0) {\pgfuseplotmark{square*}};
         \node[color=red] at (12,0) {\pgfuseplotmark{square*}};
         \node[color=red] at (18,0) {\pgfuseplotmark{square*}};
         \node[color=blue,scale=1.5] at (3,2) {\pgfuseplotmark{triangle*}};
         \node[color=blue,scale=1.5] at (9,2) {\pgfuseplotmark{triangle*}};
         \node[color=blue,scale=1.5] at (15,2) {\pgfuseplotmark{triangle*}};

         \foreach \y in {0cm, 4cm} {
            \begin{scope}[yshift=\y]
               \foreach \x in {0cm, 6cm, 12cm} {
                  \draw [xshift=\x] (0, 4) -- (3, 2) -- (6, 4);
                  \filldraw [xshift=\x] (0, 4) circle (5pt) (3, 6) circle (5pt) ;
               }
               \filldraw (18, 4) circle (5pt);
               \foreach \x in {0cm, 6cm, 12cm} {
                  \draw [xshift=\x,yshift=4cm] (0, 0) -- (3, 2) -- (6, 0);
               }
            \end{scope}
         }
         
         \draw (-3pt,4) -- (3pt,4) node[left] {$\Delta t$};
         \draw (-3pt,8) -- (3pt,8) node[left] {$2 \Delta t$};
         \draw (6, 0) node[below] {$\Delta x$};
         \draw (12, 0) node[below] {$2 \Delta x$};

         \draw (0, -0.5) node[below] {$a$};
         \draw (18, -0.5) node[below] {$b$};
        
      \end{tikzpicture}

      \caption{The domain divided into diamonds by the simple diamond
      method.  The solution, $\mathbf{z}$, is calculated at the corners of the
      diamonds.  The scheme is started using the initial
      condition, which gives the solution along the $x$ axis at the
      red squares, and
      the solution at $t = \frac{\Delta t}{2}$ (the blue triangles) which
      is calculated using a forward Euler step.  After this
      initialization the simple diamond scheme proceeds, step by step,
      to update the top of a diamond using the other three known
      points in that diamond.}
      \label{fig:simple_diamond_mesh}
   \end{figure}

To describe the simple diamond scheme consider a more detailed view of
a single diamond in figure~\eqref{fig:diamond_single}:
$\mathbf{z}_{0}^{1}$ is the solution at the top,
$\mathbf{z}_{1}^{0}$ the right most point,
$\mathbf{z}_{0}^{-1}$ the bottom,
and $\mathbf{z}_{-1}^{0}$ the left.
The point in the centre of the diamond, $\mathbf{z}_{0}^{0}$, is
defined as the average of the corner values.
   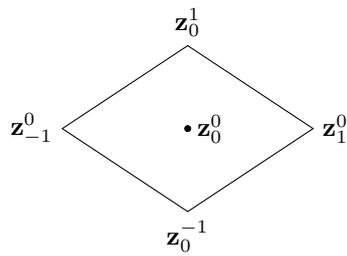
\begin{figure}
      \centering
      \begin{tikzpicture}[scale=0.55]
         \draw (-3, 0) node[left] {$\mathbf{z}_{-1}^0$} -- (0, -2) node[below] {$\mathbf{z}_{0}^{-1}$} -- (3, 0) node[right] {$\mathbf{z}_{1}^0$} -- (0, 2) node[above] {$\mathbf{z}_{0}^1$} -- cycle;
         \filldraw (0, 0) circle (2pt) node[right] {$\mathbf{z}_{0}^0$};
      \end{tikzpicture}

      \caption{A single diamond in the simple diamond scheme.  A diamond has a width of $\Delta x$ and height of
      $\Delta t$.}
      \label{fig:diamond_single}
   \end{figure}

The discrete version of Eq.~\eqref{eqn:hampde} is
\begin{gather}
   K \left( \frac{\mathbf{z}_{0}^{1} - \mathbf{z}_{0}^{-1}}{\Delta t} \right) +
   L \left( \frac{\mathbf{z}_{1}^{0} - \mathbf{z}_{-1}^{0}}{\Delta x}
   \right) = \nabla S(\mathbf{z}_{0}^{0}), \label{eqn:discretehampde} \\
   \mathbf{z}_{0}^{0} = \frac{1}{4} \left(\mathbf{z}_{0}^{1} + \mathbf{z}_{1}^{0}+ \mathbf{z}_{0}^{-1} + \mathbf{z}_{-1}^{0} \right). \label{eqn:centrept}
\end{gather}
The values $\mathbf{z}_{1}^{0}$, $\mathbf{z}_{0}^{-1}$, and $\mathbf{z}_{-1}^{0}$ are
known from the preceding step, and $\mathbf{z}_0^0$ is determined from~\eqref{eqn:centrept}, leaving
the $n$ unknowns 
$\mathbf{z}_{0}^{1}$ to be determined from the $n$ equations~\eqref{eqn:discretehampde}.  
The simple diamond scheme
solves this system of equations independently for each diamond at each time step,
then advances to the next step.   To determine the local truncation error of this scheme, substitute the
exact solution $\mathbf{z}(x + i \frac{\Delta x}{2}, t+ j \frac{\Delta t}{2})$ for $\mathbf{z}_i^j$ into Eq.~\eqref{eqn:discretehampde}
and expand in Taylor series:
\begin{gather}
   K\left(\mathbf{z}_t + \frac{\Delta t^2}{4} \mathbf{z}_{ttt} + \BigO{\Delta t^3}\right) + L\left(\mathbf{z}_x + \frac{\Delta x^2}{4} \mathbf{z}_{xxx} + \BigO{\Delta x^3}\right) = \nabla S(\mathbf{z}) \\
   \Rightarrow K\mathbf{z}_t + L\mathbf{z}_x = \nabla S(\mathbf{z}) + \BigO{\Delta t^2 + \Delta x^2};
\end{gather}
thus the order is $\BigO{\Delta t^2 + \Delta x^2}$.

For the one-dimensional wave equation defined by
Eqs.~\eqref{eqn:hampde} and~\eqref{eqn:1dwaveeqn}
the simple diamond scheme becomes
\begin{align}
\frac{{u_{1}^{0}-u_{-1}^{0}}}{\Delta x} &=\frac{{w_{0}^{-1}+w_{1}^{0}+w_{-1}^{0}+w_{0}^{1}}}{4}\label{eqn:wave2discrete},\\
\frac{u_{0}^{1}-u_{0}^{-1}}{\Delta t} & =\frac{v_{0}^{-1}+v_{1}^{0}+v_{-1}^{0}+v_{0}^{1}}{4}\label{eqn:wave1discrete},\\
\frac{v_{0}^{1}-v_{0}^{-1}}{\Delta t} & =\frac{w_{1}^{0}-w_{-1}^{0}}{\Delta x} + f(u_{0}^{0})\label{eqn:wave3discrete} .
\end{align}
At each time step, for each diamond, Eq.~\eqref{eqn:wave2discrete} is first solved
to give the new $w_{0}^{1}$.  Then Eq.~\eqref{eqn:wave1discrete} is solved for
$v_0^1$ and substituted into~\eqref{eqn:wave3discrete} to give a scalar equation of the form
\begin{equation}
\label{eqn:simplediamondsolve}
u_0^1 = C + \frac{(\Delta t)^2}{4} f(u_0^0)
\end{equation}
for $u_0^1$ where $C$ depends on the known data. This equation has a solution $u_0^1 = C +\BigO{(\Delta t)^2}$ for sufficiently small $\Delta t$ when $f$ is Lipschitz.
Thus, although the scheme is implicit, it is only {\em locally} implicit within each cell; a set of $N$ uncoupled scalar equations is typically much easier to solve than a system of $N$ coupled equations.

\begin{prop} \label{prop:sds_conservation_law}
   The simple diamond scheme shown in
   Eq.~\eqref{eqn:discretehampde} satisfies the discrete
   conservation law
\begin{multline*}
   \frac{1}{4\Delta t} \left(( d \mathbf{z}_{-1}^{0}+ d \mathbf{z}_{0}^{1} + d \mathbf{z}_{1}^{0} ) \wedge K d \mathbf{z}_{0}^{1} - ( d \mathbf{z}_{-1}^{0}+ d \mathbf{z}_{0}^{-1} + d \mathbf{z}_{1}^{0} ) \wedge K d \mathbf{z}_{0}^{-1}\right) \\
 + \frac{1}{4\Delta x} \left(( d \mathbf{z}_{0}^{1}+ d \mathbf{z}_{1}^{0} + d \mathbf{z}_{0}^{-1} ) \wedge L d \mathbf{z}_{1}^{0} - ( d \mathbf{z}_{0}^{1}+ d \mathbf{z}_{-1}^{0} + d \mathbf{z}_{0}^{-1} ) \wedge L d \mathbf{z}_{-1}^{0}\right) = 0
\end{multline*}
\end{prop}
\begin{proof}
Take the exterior derivative and apply $d \mathbf{z}_0^0 \wedge $ on the left of
Eq.~\eqref{eqn:discretehampde} to give
\begin{multline*}
   \tfrac{1}{4\Delta t}( d \mathbf{z}_{-1}^{0}+ d \mathbf{z}_{0}^{1} + d \mathbf{z}_{1}^{0} + d \mathbf{z}_{0}^{-1} ) \wedge K ( d \mathbf{z}_{0}^{1} - d \mathbf{z}_{0}^{-1} ) \\ +
   \tfrac{1}{4\Delta x}( d \mathbf{z}_{0}^{1}+ d \mathbf{z}_{1}^{0} + d \mathbf{z}_{0}^{-1} + d \mathbf{z}_{-1}^{0} ) \wedge L ( d \mathbf{z}_{1}^{0} - d \mathbf{z}_{-1}^{0} ) =0
\end{multline*}
Expanding and simplifying leads to the result.
\end{proof}

Although a recursive implementation of the simple diamond scheme does
not run particularly fast, it does clearly define the algorithm:
\greybox{%
   Let $N$ be the number of diamonds across the domain $[a,b]$, and
   \texttt{z[i,j]} approximate $z(a+i\tfrac{\Delta x}{2},
   j\tfrac{\Delta t}{2})$, where
   $i \in \{0, 1, \ldots, 2N-1\}, j \in \{0,1,  \ldots \}$.
   \\The functions \texttt{z0(x)} and $\mathtt{z0_t(x)}$ give the initial conditions.
   \\The function \texttt{fun(i, j)} defines \texttt{z[i,j]} recursively:
   \\\texttt{fun(i,j)
   \\\hspace*{3ex}if i is -1 then i = 2N-1
   \\\hspace*{3ex}if i is 2N then i = 0
   \\\hspace*{3ex}if j is 0 then
   \\\hspace*{3ex}\hspace*{3ex}return z0(a + i$\Delta$x/2)
   \\\hspace*{3ex}else if j is 1 then
   \\\hspace*{3ex}\hspace*{3ex}return z0(a + i$\Delta$x/2) + $\Delta$t/2 $\mathtt{z0_t}$(a + i$\Delta$x/2)
   \\\hspace*{3ex}else
   \\\hspace*{3ex}\hspace*{3ex}Use a numerical solver to find z such that
   \begin{multline*}
      \hspace*{3ex}\hspace*{3ex}K \left( \frac{\mathtt{z} - \mathtt{fun(i,j-2)}}{\Delta t} \right) + L \left( \frac{\mathtt{fun(i-1,j-1)} - \mathtt{fun(i-1,j-1)}}{\Delta x} \right) \\= \nabla S( \tfrac{\mathtt{z}+\mathtt{fun(i,j-2)}+\mathtt{fun(i-1,j-1)}+\mathtt{fun(i-1,j-1)}}{4})
   \end{multline*}
   \\\hspace*{3ex}\hspace*{3ex}return z
   }
}
The above algorithm can be considerably sped up by caching the
results of the calculation of \texttt{z}.

\subsection{Numerical test}
\label{sec:sinegordon}
As a test the Sine--Gordon equation, $u_{tt}-u_{xx}=-\sin(u)$ will be solved using the simple diamond scheme.
An exact solution is the so-called \emph{breather},
\begin{equation} \label{eqn:breather}
u(x,t) = 4 \arctan \left( \frac{\sin\left(\frac{t}{\sqrt{2}}\right)}{ \cosh\left(\frac{x}{\sqrt{2}}\right)} \right).
\end{equation}
The domain is taken significantly large, $\left[-30, 30\right]$, so
the solution can be assumed periodic.  The initial conditions are
calculated using the exact solution.
The error is the discrete 2-norm of $u$,
\begin{equation} \label{eqn:errornorm}
E^2 = \frac{b - a}{N} \sum_i^N \left(\tilde{u}_i - u(a + i \Delta x, T) \right)^2.
\end{equation}

Figure~\eqref{fig:error_rsimple_sinegordon_periodic_s2_internal0} shows the error
of the simple diamond scheme as $\Delta t$ is reduced while keeping the Courant
number $\frac{\Delta t}{\Delta x} = \frac{1}{2}$.  The final run time,
$T$, of the
scheme was chosen so that the coarsest run, that with the largest $\Delta t$,
ran for two steps.  It is apparent that for this problem, the method is of order 2.
\begin{figure}
   \centering
   \includegraphics[width=3in]{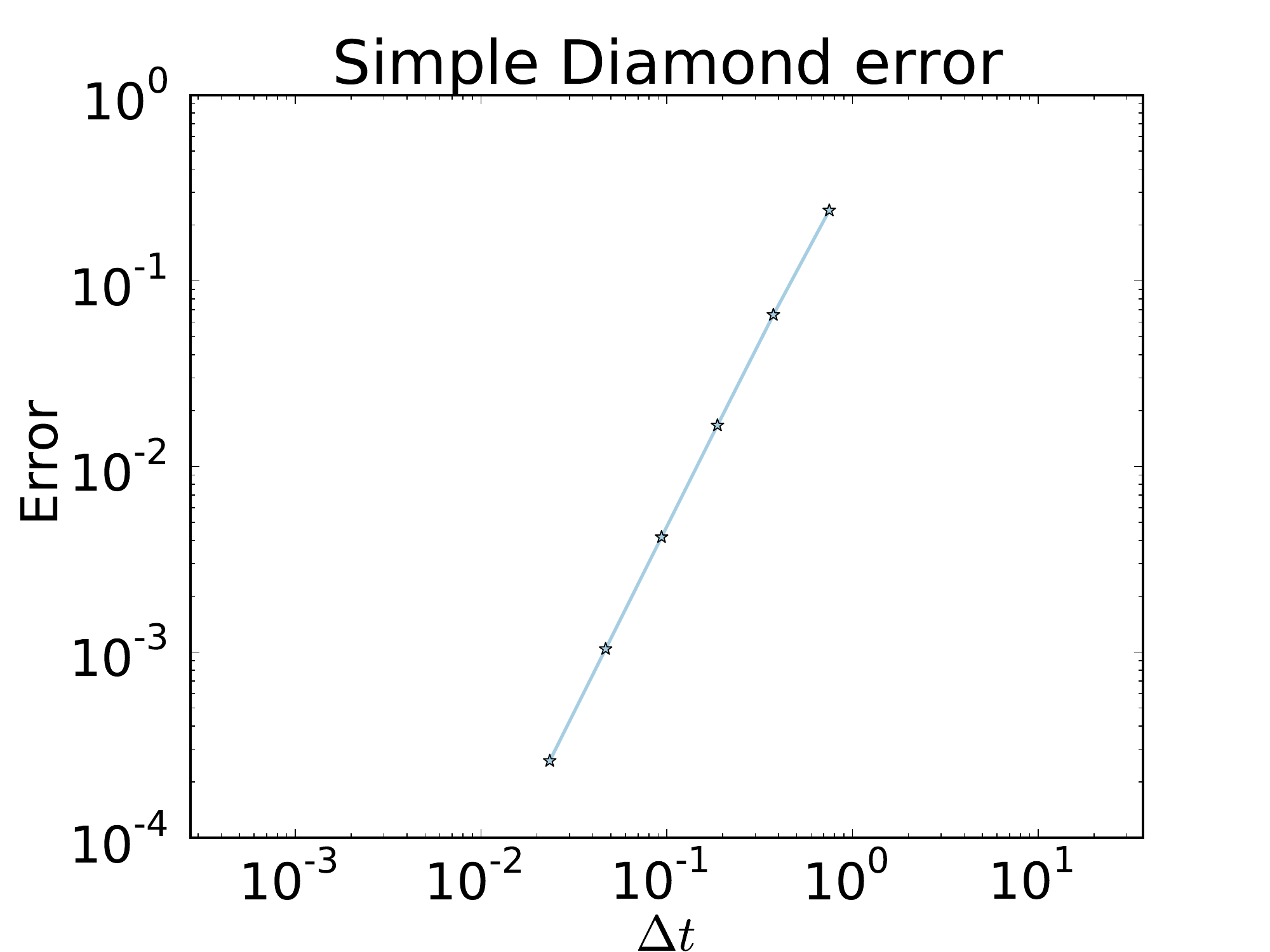}
   \input{\inputdir/error_rsimple_sinegordon_periodic_s2_internal0_caption}
   \label{fig:error_rsimple_sinegordon_periodic_s2_internal0}
\end{figure}

\section{The diamond scheme}
The diamond scheme refines the simple diamond scheme discretization by
using the multisymplectic Runge-Kutta
collocation method given by Reich~\cite{reich} within each diamond.
It is easier to apply this method to a square that is aligned with the
axes, so the first step is to transform the $(x,t)$ coordinate space.
Each diamond in figure~\eqref{fig:simple_diamond_mesh} is transformed to a square of side length one using the linear transformation $T$ defined by
\begin{equation} \label{eqn:xttilde}
T\colon\quad   \tilde{x} = \tfrac{1}{\Delta x} x + \tfrac{1}{\Delta t} t \quad \mathrm{and} \quad \tilde{t} = -\tfrac{1}{\Delta x} x + \tfrac{1}{\Delta t} t.
\end{equation}
Because Eq.~\eqref{eqn:hampde} has no dependence on $x$ or $t$ it
doesn't matter where the square is located in $(\tilde{x},\tilde{t})$
space, so the same transformation can be used for all the diamonds.
Let $\tilde{\mathbf{z}}( \tilde{x}, \tilde{t}) = \mathbf{z}(x, t)$.  By the chain rule
\begin{equation} \label{eqn:zxandzt}
   \mathbf{z}_x = \tilde{\mathbf{z}}_{\tilde{x}} \tfrac{1}{\Delta x} - \tilde{\mathbf{z}}_{\tilde{t}} \tfrac{1}{\Delta x}  \quad \mathrm{and} \quad
   \mathbf{z}_t = \tilde{\mathbf{z}}_{\tilde{x}} \tfrac{1}{\Delta t} + \tilde{\mathbf{z}}_{\tilde{t}} \tfrac{1}{\Delta t},
\end{equation}
so
\begin{align*}
   K \mathbf{z}_t + L \mathbf{z}_x &= K \left( \tilde{\mathbf{z}}_{\tilde{x}} \tfrac{1}{\Delta t} + \tilde{\mathbf{z}}_{\tilde{t}} \tfrac{1}{\Delta t} \right) + L \left( \tilde{\mathbf{z}}_{\tilde{x}} \tfrac{1}{\Delta x} - \tilde{\mathbf{z}}_{\tilde{t}} \tfrac{1}{\Delta x} \right), \\
   &= \left(\tfrac{1}{\Delta t} K -\tfrac{1}{\Delta x} L \right) \tilde{\mathbf{z}}_{\tilde{t}} + \left( \tfrac{1}{\Delta t} K + \tfrac{1}{\Delta x} L \right) \tilde{\mathbf{z}}_{\tilde{x}},
\end{align*}
thus Eq.~\eqref{eqn:hampde} transforms to
\begin{equation}
   \tilde{K} \tilde{\mathbf{z}}_{\tilde{t}} + \tilde{L} \tilde{\mathbf{z}}_{\tilde{x}} = \nabla S(\tilde{\mathbf{z}}),
\end{equation}
where
\begin{equation} \label{eqn:KLtilde}
   \tilde{K} = \tfrac{1}{\Delta t} K - \tfrac{1}{\Delta x} L \quad \mathrm{and} \quad \tilde{L} = \tfrac{1}{\Delta t} K + \tfrac{1}{\Delta x} L.
\end{equation}

\subsubsection*{Outline of method}

Figure~\eqref{fig:domain} illustrates the diamond scheme for a sample initial-boundary value problem on $[a,b]\times\mathbb{R}^+$ with periodic boundary conditions.
   \begin{figure}
      \centering
      \begin{tikzpicture}[scale=0.35]

         \draw[<->,thick] (0, 11) -- (0, 0) -- (18, 0) -- (18, 11);
         \draw[<->,thick] (-1, 0) -- (19, 0);
         \draw[<->,thick] (-1, 0) -- (19, 0);
         
         \draw[color=blue,style=dashed] (-3, 2) -- (0, 0);
         \foreach \x in {0cm, 6cm, 12cm} {
            \draw[color=blue] [xshift=\x] (0, 0) -- (3, 2) -- (6, 0);
         }
        
         \foreach \y in {0cm, 4cm} {
            \begin{scope}[yshift=\y]
               \foreach \x in {0cm, 6cm, 12cm} {
                  \draw[color=green,dash pattern=on 4pt off 3pt on 1.0pt off 3pt] [xshift=\x] (-3, 2) -- (0, 4) -- (3, 2);
               }
               \draw[color=green,style=dashed] (15, 2) -- (18, 4);
               \draw[color=red,style=dashed] (-3, 6) -- (0, 4);
               \foreach \x in {0cm, 6cm, 12cm} {
                  \draw[color=red,dash pattern=on 4pt off 3pt on 1.0pt off 3pt on 1.0pt off 3pt] [xshift=\x,yshift=4cm] (0, 0) -- (3, 2) -- (6, 0);
               }
            \end{scope}
         }
         
         \draw (-3pt,4) -- (3pt,4) node[left] {$\Delta t$};
         \draw (0, 10.5) node[left] {$t$};
         \draw (-3pt,8) -- (3pt,8) node[left] {$2 \Delta t$};
         \draw (12,2) node {$\Delta x$};
         \draw[<-] (9.5,2) -- (11.3,2);
         \draw[->] (12.7,2) -- (14.5,2);
         \draw (0, -0.5) node[below] {$a$};
         \draw (9, -0.5) node[below] {$x$};
         \draw (18, -0.5) node[below] {$b$};
        
         \draw[->,blue] (7.5,3) -- (7.0,4.666);
         \draw[->,blue] (7.5,3) -- (10.5,5);
         \draw[->,blue] (10.5,3) -- (7.5,5);
         \draw[->,blue] (10.5,3) -- (11.0,4.666);

         \draw[<-] (17,1) -- (20,1) node[right] {forward Euler};
         \draw[<-] (17,3) -- (20,3) node[right] {$1^\mathrm{st}$ half-step};
         \draw[<-] (17,5) -- (20,5) node[right] {finish $1^\mathrm{st}$ step};
         \draw[<-] (17,7) -- (20,7) node[right] {$2^\mathrm{nd}$ half-step};
         \draw[<-] (17,9) -- (20,9) node[right] {finish $2^\mathrm{nd}$ step};
      
      \end{tikzpicture}

      \caption{Information flows upwards as indicated by the solid blue 
      arrows for a typical diamond.  The solution, $\mathbf{z}$, is initialized on
      the solid blue zig-zag line using a forward Euler step.
      A step of
      the diamond scheme consists of two half-steps.  The first half-step
      calculates $\mathbf{z}$ along the green dash-dot line, which by periodicity is
      extended to the dashed line to the right.  The second half step uses
      the green dash-dot line to calculate the red dash-double-dotted line, which again by
      periodicity is extended to the left-hand dashed segment.}
      \label{fig:domain}
   \end{figure}
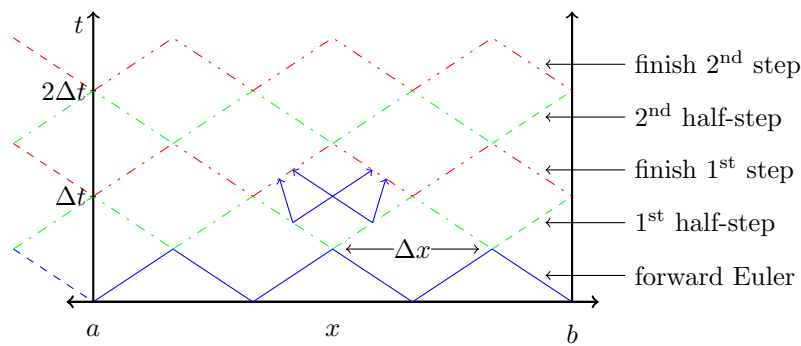
The solution
$\mathbf{z}$ is calculated at grid points located on the solid diamond edges; dashed edges indicate where values are
inferred by periodicity.  Information follows from the bottom
left and right edges of a diamond to the top left and right edges of
the same diamond.   Initial conditions are given for
$\mathbf{z}$ along the bottom edges of the first row of diamonds (the
first solid blue zig-zag line in the figure).   This is extended using
periodicity beyond the left hand boundary (the blue dashed line).  A
step of the diamond scheme consists of two half-steps.  The first half-step
calculates $\mathbf{z}$ along the top edges of the first row of
diamonds (green dash-dot zig-zag), which by periodicity is extended to the
right-hand boundary (the green dashed line).  The second half-step uses
values on the top edges of the first row of diamonds (green dash-dot line) to
calculate the new values of $\mathbf{z}$ on the top edges of the
second row (red dash-double-dotted line).  Again by periodicity the values from the top
right edge of the right most diamond are copied outside the left hand
boundary of the domain (the red dashed segment).  Another step can be
performed now using the red dash-double-dotted zig-zag as initial data (the very right
hand line segment is not used except to provide values for the dashed
line).

\subsubsection*{Updating one diamond}

Let $(A,b,c)$ be the parameters of an $r$-stage Runge--Kutta method. In what
follows, we will take the method to be the Gauss Runge--Kutta method.
Figure~\eqref{fig:single_diamond} shows a
diamond with $r = 3$, and its transformation to the unit square.
The square contains $r \times
r$ internal grid points, as determined by the Runge-Kutta coefficients $c$,
 and internal stages $\mathbf{Z}_i^j$, which are analogous to the usual
internal grid points and stages in a Runge-Kutta method. The internal stages also carry 
the variables $\mathbf{X}_i^j$ and $\mathbf{T}_i^j$ which approximate $z_x$ and $z_t$, respectively, at the internal stages.

The dependent variables of the method are the values of $z$ at the edge grid points
To be able to distinguish the internal
edge points from all the edge points let $I$ be the set of indices
$\left\{1, \ldots, r\right\}$.  Then, for example, $\tilde{\mathbf{z}}^\mathrm{b}_I$ refers
to $\tilde{\mathbf{z}}^\mathrm{b}_i, i = 1 \ldots r $.
If the $I$ qualifier does not appear
then the left or bottom most
corner is included, for example $\tilde{\mathbf{z}}^\mathrm{b}$ refers to the points
$\tilde{\mathbf{z}}^\mathrm{b}_i, i = 0 \ldots r $, but does not include
$\tilde{\mathbf{z}}^\mathrm{b}_{r+1}$ which is 
$\tilde{\mathbf{z}}_\mathrm{r}^0$.
Note $\tilde{\mathbf{z}}^\mathrm{b}_0 = \tilde{\mathbf{z}}_\ell^0$ (which is also the same as $\mathbf{z}_0^{-1}$).
   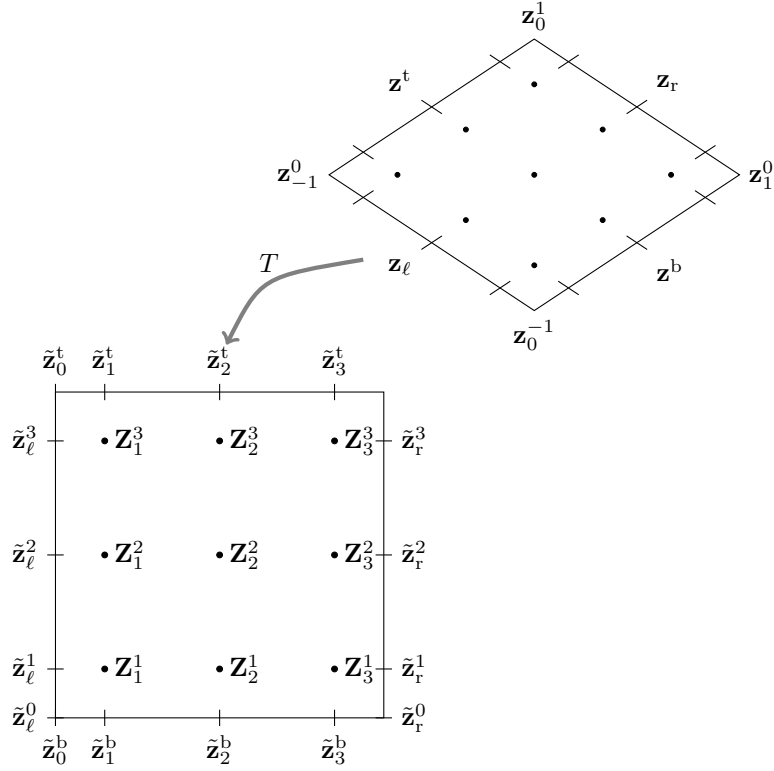
\begin{figure}
      \centering
      \begin{tikzpicture}[scale=0.9]

         \draw (-3, 0) -- (0, -2) -- (3, 0) -- (0, 2) -- cycle;
         \draw (0, 2) node[above]{$\mathbf{z}_0^1$};
         \draw (0, -2) node[below]{$\mathbf{z}_0^{-1}$};
         \draw (3, 0) node[right]{$\mathbf{z}_1^{0}$};
         \draw (-3, 0) node[left]{$\mathbf{z}_{-1}^0$};

         \draw [xshift=0.5cm,yshift=-1.667cm] (-0.15, 0.1) -- (0.15, -0.1); 
         \draw [xshift=1.5cm,yshift=-1cm] (-0.15, 0.1) -- (0.15, -0.1) node[below right]{$\mathbf{z}^\mathrm{b}$};
         \draw [xshift=2.5cm,yshift=-0.333cm] (-0.15, 0.1) -- (0.15, -0.1); 

         \draw [xshift=-2.5cm,yshift=-0.333cm] (-0.15, -0.1) -- (0.15, 0.1); 
         \draw [xshift=-1.5cm,yshift=-1cm]  (-0.15, -0.1) node[below left]{$\mathbf{z}_\ell$} -- (0.15, 0.1);
         \draw [xshift=-0.5cm,yshift=-1.667cm] (-0.15, -0.1) -- (0.15, 0.1); 

         \begin{scope}[xshift=-3cm,yshift=2cm]
         \draw [xshift=0.5cm,yshift=-1.667cm] (-0.15, 0.1) -- (0.15, -0.1); 
         \draw [xshift=1.5cm,yshift=-1cm] (-0.15, 0.1) node[above left]{$\mathbf{z}^\mathrm{t}$} -- (0.15, -0.1);
         \draw [xshift=2.5cm,yshift=-0.333cm] (-0.15, 0.1) -- (0.15, -0.1); 
         \end{scope}

         \begin{scope}[xshift=3cm,yshift=2cm]
         \draw [xshift=-2.5cm,yshift=-0.333cm] (-0.15, -0.1) -- (0.15, 0.1); 
         \draw [xshift=-1.5cm,yshift=-1cm]  (-0.15, -0.1) -- (0.15, 0.1)   node[above right]{$\mathbf{z}_\mathrm{r}$};
         \draw [xshift=-0.5cm,yshift=-1.667cm] (-0.15, -0.1) -- (0.15, 0.1); 
         \end{scope}

         \filldraw   (-2, 0) circle (1pt) 
         (-1, 0.667) circle (1pt) 
         (0, 1.333) circle (1pt) 
         (-1, -0.667) circle (1pt) 
         (0, 0) circle (1pt) 
         (1, 0.667) circle (1pt) 
         (0, -1.333) circle (1pt) 
         (1, -0.667) circle (1pt) 
         (2, 0) circle (1pt) ;

         \draw[->,gray,ultra thick] (-2.5,-1.25) .. controls (-4.0,-1.5) .. (-4.5, -2.5) node[pos=0.5,above,black] {$T$};

         \begin{scope}[xshift=-7cm, yshift=-8cm, scale=1.2]

         \draw (0, 0) -- (4, 0) -- (4, 4) -- (0, 4) -- cycle;

         \foreach \x / \xtext in {0, 0.6/1, 2, 3.4/3}
            \draw [xshift=\x cm,yshift=0cm] (0, 0.1) -- (0, -0.1) node[below]{$\tilde{\mathbf{z}}^\mathrm{b}_\xtext$};
         \foreach \x / \xtext in {0, 0.6/1, 2, 3.4/3}
            \draw [xshift=\x cm,yshift=4cm] (0, 0.1) node[above]{$\tilde{\mathbf{z}}^\mathrm{t}_\xtext$} -- (0, -0.1);
         \foreach \y / \ytext in {0, 0.6/1, 2, 3.4/3}
            \draw [xshift=0cm,yshift=\y cm] (-0.1, 0)  node[left]{$\tilde{\mathbf{z}}_\ell^\ytext$} -- (0.1, 0);
         \foreach \y / \ytext in {0, 0.6/1, 2, 3.4/3}
            \draw [xshift=4cm,yshift=\y cm] (-0.1, 0) -- (0.1, 0) node[right]{$\tilde{\mathbf{z}}_\mathrm{r}^\ytext$};
         \foreach \y / \ytext in {0.6/1, 2, 3.4/3} {
            \foreach \x / \xtext in {0.6/1, 2, 3.4/3} {
               \filldraw (\x, \y) circle (1pt) node[right]{$\mathbf{Z}_\xtext^\ytext$};
            }
         }
            \end{scope}

      \end{tikzpicture}

      \caption{The diamond transformed by a linear transformation,
      $T$, to the unit square.  The square
      contains $r \times r$ ($r=3$ in this example) internal stages, $\mathbf{Z}_i^j$.
      The solution is known along the bottom and left hand sides.
      The method proceeds as two sets of $r$ Gauss Runge-Kutta $r$-step
      methods: internal stage values, $\mathbf{Z}_i^j,\ \mathbf{X}_i^j,\ \mathbf{T}_i^j$, are calculated, then the
      right and top updated.}
      \label{fig:single_diamond}
   \end{figure}
The Runge--Kutta discretization is
\begin{align}
   \mathbf{Z}_i^j &= \tilde{\mathbf{z}}_\ell^j + \sum_{k=1}^r a_{ik} \mathbf{X}_k^j, \label{eqn:msZ1} \\
   \mathbf{Z}_i^j &= \tilde{\mathbf{z}}^\mathrm{b}_i + \sum_{k=1}^r a_{jk} \mathbf{T}_i^k, \label{eqn:msZ2} \\
   \nabla S(\mathbf{Z}_i^j) &= \tilde{K} \mathbf{T}_i^j + \tilde{L} \mathbf{X}_i^j, \label{eqn:msZ3}
\end{align}
together with the update equations
\begin{align}
   \tilde{\mathbf{z}}_\mathrm{r}^i &= \tilde{\mathbf{z}}_\ell^i + \sum_{k=1}^r b_k  \mathbf{X}_k^i, \label{eqn:msupdate1} \\
   \tilde{\mathbf{z}}^\mathrm{t}_i &= \tilde{\mathbf{z}}^\mathrm{b}_i + \sum_{k=1}^r b_k\mathbf{T}_i^k, \label{eqn:msupdate2}
\end{align}
for $i,j \in I$.  The
$\tilde{\mathbf{z}}_\ell^I$ and
$\tilde{\mathbf{z}}^\mathrm{b}_I$ are known.
Eqs.~\eqref{eqn:msZ1}--\eqref{eqn:msZ3} are first solved for the
internal stage values $\mathbf{Z}_i^j$, $\mathbf{X}_i^j$, and $\mathbf{T}_i^j$, then Eqs.~\eqref{eqn:msupdate1}
and~\eqref{eqn:msupdate2} are used to calculate 
$\tilde{\mathbf{z}}^\mathrm{t}_I$ and
$\tilde{\mathbf{z}}_\mathrm{r}^I$.  
Eqs.~\eqref{eqn:msZ1}--\eqref{eqn:msZ3} 
are $3r^2$
equations in $3r^2$ unknowns 
$\mathbf{Z}$, $\mathbf{X}$, and $\mathbf{T}$.
Eqs.~\eqref{eqn:msZ1} and~\eqref{eqn:msZ2} are linear in 
$\mathbf{X}$ and $\mathbf{T}$. 
Thus in practice the method requires solving a set of $r^2n$ nonlinear
equations for $\mathbf{Z}$ in each diamond.

The method does not use values at the corners. However, if solutions
are wanted at the corners then  the method can be extended by: 
allowing $j = 0$ in
Eq.~\eqref{eqn:msZ1} and associating $\mathbf{Z}_i^0$ with
$\tilde{\mathbf{z}}^\mathrm{b}_i$; and
allowing $i = 0$ in
Eq.~\eqref{eqn:msZ2} and associating $\mathbf{Z}_0^j$ with
$\tilde{\mathbf{z}}_\ell^j$.  This extension gives equations for
$\mathbf{X}_k^0$ and 
$\mathbf{T}_0^k$, which can be used in
the update Eqs.~\eqref{eqn:msupdate1} and~\eqref{eqn:msupdate2}
which are extended by allowing $i=0$.
On this extended domain
Eqs.~\eqref{eqn:msZ1}--\eqref{eqn:msZ3} are
$2r(r+1) + r^2$  equations.  This is because: Eq.~\eqref{eqn:msZ1}
is extended onto the bottom boundary ($r(r+1)$ equations), 
Eq.~\eqref{eqn:msZ2} onto the left boundary ($r(r+1)$ equations),
but there is no need to extend Eq.~\eqref{eqn:msZ3} onto either
boundaries because $\mathbf{T}$ was not extended onto the bottom
boundary, and $\mathbf{X}$ was not
extended onto the left boundary.  The
number of unknowns is $2r(r+1) + r^2$, so again there are the same
number of equations as unknowns.
Corner points are shared by two adjacent diamonds, and typically
$\tilde{\mathbf{z}}^\mathrm{b}_{r+1}\ne\tilde{\mathbf{z}}_\ell^{r+1}$. In
practice the mean of these two approximations is used.

Here is a summary of the diamond scheme algorithm:
\greybox{%
   Let $\tilde{\mathtt{z}}$ and $\tilde{\mathtt{z}}_n$ be $N(2r+1)$ length
   vectors with each element in $\mathbb{R}^n$.  These vectors contain
   $\tilde{\mathbf{z}}$ values for two particular edges of each diamond.  Each
   of the two edges has $r$ nodes, plus there is the value at the
   bottom, hence $2r+1$ values per $N$ diamonds.
   \begin{enumerate}
      \item Using Forward Euler, initialize $\tilde{\mathtt{z}}$.  It now contains $\tilde{\mathbf{z}}$ for the blue zig-zag at the bottom in figure~\eqref{fig:domain}.
      \item The half-step.  For each diamond:
      \begin{enumerate}
         \item Associate $\tilde{\mathbf{z}}_\ell$ and $\tilde{\mathbf{z}}^\mathrm{b}$ with the correct values in $\tilde{\mathtt{z}}$ (periodicity is used at the edges).
         \item Solve Eqs.~\eqref{eqn:msZ1}--\eqref{eqn:msZ3}.
         \item Use Eqs.~\eqref{eqn:msupdate1} and~\eqref{eqn:msupdate2} to find $\tilde{\mathbf{z}}_\mathrm{r}$ and $\tilde{\mathbf{z}}^\mathrm{t}$
         \item Associate $\tilde{\mathbf{z}}_\mathrm{r}$ and $\tilde{\mathbf{z}}^\mathrm{t}$ with the current diamond's section of $\tilde{\mathtt{z}}_n$ (periodicity is used at the edges).
      \end{enumerate}
      \item $\tilde{\mathtt{z}} = \tilde{\mathtt{z}}_n$.  Now $\tilde{\mathtt{z}}$ contains $\tilde{\mathbf{z}}$ values for the second/green zig-zag in figure~\eqref{fig:domain}.
      \item Perform step 2 again.
      \item $\tilde{\mathtt{z}} = \tilde{\mathtt{z}}_n$.  Now $\tilde{\mathtt{z}}$ contains $\tilde{\mathbf{z}}$ values for the third/red zig-zag in figure~\eqref{fig:domain}.
      \item If final time not reached go to step 2.
   \end{enumerate}
}

\begin{thm}
For the multi-Hamiltonian one dimensional wave equation defined 
by Eqs.~\eqref{eqn:hampde} and~\eqref{eqn:1dwaveeqn}
with the following conditions:
\begin{itemize}
   \item $f$ is Lipschitz with constant $L$,
   \item The matrix
      \begin{equation*}
         B = (1-\lambda^2)(  I \otimes A^{-2}  ) + 2(1+\lambda^2) (A^{-1} \otimes A^{-1}) + (1-\lambda^2) (  A^{-2}  \otimes I ), 
      \end{equation*}
      where $A$ is the matrix of coefficients of the underlying Runge--Kutta scheme and $\lambda = \frac{\Delta t}{\Delta x}$ is the Courant number,
      exists and is invertible, and,
   \item $(\Delta t)^2 < \frac{1}{L \norm{B^{-1}}_\infty}$,
\end{itemize}
Eqs.~\eqref{eqn:msZ1}--~\eqref{eqn:msZ3} are solvable, and thus the diamond scheme is well defined.
\end{thm}
\begin{proof}
Eqs.~\eqref{eqn:msZ1} and~\eqref{eqn:msZ2} relate components of
matrices and tensors.  Writing these equations in tensor form, then
multiplying on the left by $A^{-1}$ gives expressions for $
\mathbf{X}_i^j$ and $\mathbf{T}_i^j$.  Substituting
these into Eq.~\eqref{eqn:msZ3} gives
\begin{equation*}
   \nabla S(\mathbf{Z}_i^j) = \tilde{K} \sum_{k=1}^r m_{jk} (\mathbf{Z}_i^k - \mathbf{z}^\mathrm{b}_i ) + \tilde{L} \sum_{k=1}^r m_{ik} (\mathbf{Z}_k^j - \mathbf{z}_\ell^k ),
\end{equation*}
where $m_{ij}$ are the elements of $A^{-1}$, and the tildes on the $\mathbf{z}$ values have been dropped for clarity.
Using Eq.~\eqref{eqn:KLtilde} for 
$\tilde{K}$ and $\tilde{L}$, and adopting the summation convention, this becomes
\begin{equation*}
   \begin{pmatrix} -f(u_i^j)\\v_i^j\\-w_i^j\end{pmatrix} =
      \begin{pmatrix} 0 & \frac{-1}{\Delta t} & \frac{-1}{\Delta x} \\ \frac{1}{\Delta t} & 0 & 0 \\ \frac{1}{\Delta x} & 0 &0 \end{pmatrix} 
                        m_{jk} (\mathbf{Z}_i^k - \mathbf{z}^\mathrm{b}_i ) + 
      \begin{pmatrix} 0 & \frac{-1}{\Delta t} & \frac{1}{\Delta x} \\ \frac{1}{\Delta t} & 0 & 0 \\ \frac{-1}{\Delta x} & 0 &0 \end{pmatrix} 
                        m_{ik} (\mathbf{Z}_k^j - \mathbf{z}_\ell^k ).
\end{equation*}
The solution for $v_i^j$ and $w_i^j$,
\begin{align*}
   v_i^j &= \frac{1}{\Delta t} m_{jk}(u_i^k - u^\mathrm{b}_i ) + \frac{1}{\Delta t} m_{ik}(u_k^j - u_\ell^k ),\\
   w_i^j &= \frac{-1}{\Delta x} m_{jk}(u_i^k - u^\mathrm{b}_i ) + \frac{1}{\Delta x} m_{ik}(u_k^j - u_\ell^k ),
\end{align*}
is substituted into the equation for $u_i^j$, which after
simplification gives
\begin{equation*}
   (1-\lambda^2)m_{jk}m_{kp} u_i^p + 2(1+\lambda^2)m_{jk}m_{ip}u_p^k+(1-\lambda^2)m_{ik}m_{kp}u_p^j = \mathbf{b} + (\Delta t)^2 f(u_i^j),
\end{equation*}
where the vector $\mathbf{b}$
is a constant term depending on $A^{-1}$ and the initial data
$\mathbf{z}_\ell$ and $\mathbf{z}^\mathrm{b}$.
Let $\mathbf{u} = (u_1^1, u_1^2, \ldots, u_1^r, u_2^1, \ldots, u_3^1, \ldots, u_r^r)$ and $f(\mathbf{u})
= (f(u_1^1), \ldots, f(u_r^r))$, then this simplifies to
\begin{equation}
   B \mathbf{u} = \mathbf{b} +  \Delta t^2f(\mathbf{u}),  \label{eqn:solvableu}
\end{equation}
where $B$ is given in the conditions of the theorem.
To complete the proof it must be shown that this equation has a solution.
Because $B$ is invertible,
$G(\mathbf{u}) = B^{-1}(\mathbf{b} +  \Delta t^2f(\mathbf{u}))$ exists.  Consider $G$ applied to the two points $\mathbf{u}^1$ and $\mathbf{u}^2$
\begin{align*}
   \norm{G(\mathbf{u}^1) - G(\mathbf{u}^2)}_\infty &= \Delta t^2 \norm{B^{-1}(f(\mathbf{u}^1)-f(\mathbf{u}^2))}_\infty \\
                                                   &\le \Delta t^2 \norm{B^{-1}}_\infty L \norm{\mathbf{u}^1 - \mathbf{u}^2}_\infty.
\end{align*}
By the contraction mapping theorem and the condition on $\Delta t$, $G$ must have a fixed point $\mathbf{u} = G(\mathbf{u})$, thus Eq.~\eqref{eqn:solvableu} has a solution.
\end{proof}

For a particular Runge--Kutta method it is straightforward to calculate the matrix
$B$ and determine the conditions on $\lambda$ that lead to solvability.
Figure~\eqref{fig:min_sv_B_versus_c} shows that for Gauss Runge--Kutta and $r=1,\dots,5$ and $\lambda
\in [0, 1]$, the minimum singular value of $B$ is nonzero.  This calculation can
be performed for larger $r$.
\begin{figure}
   \centering
   \includegraphics[width=3in]{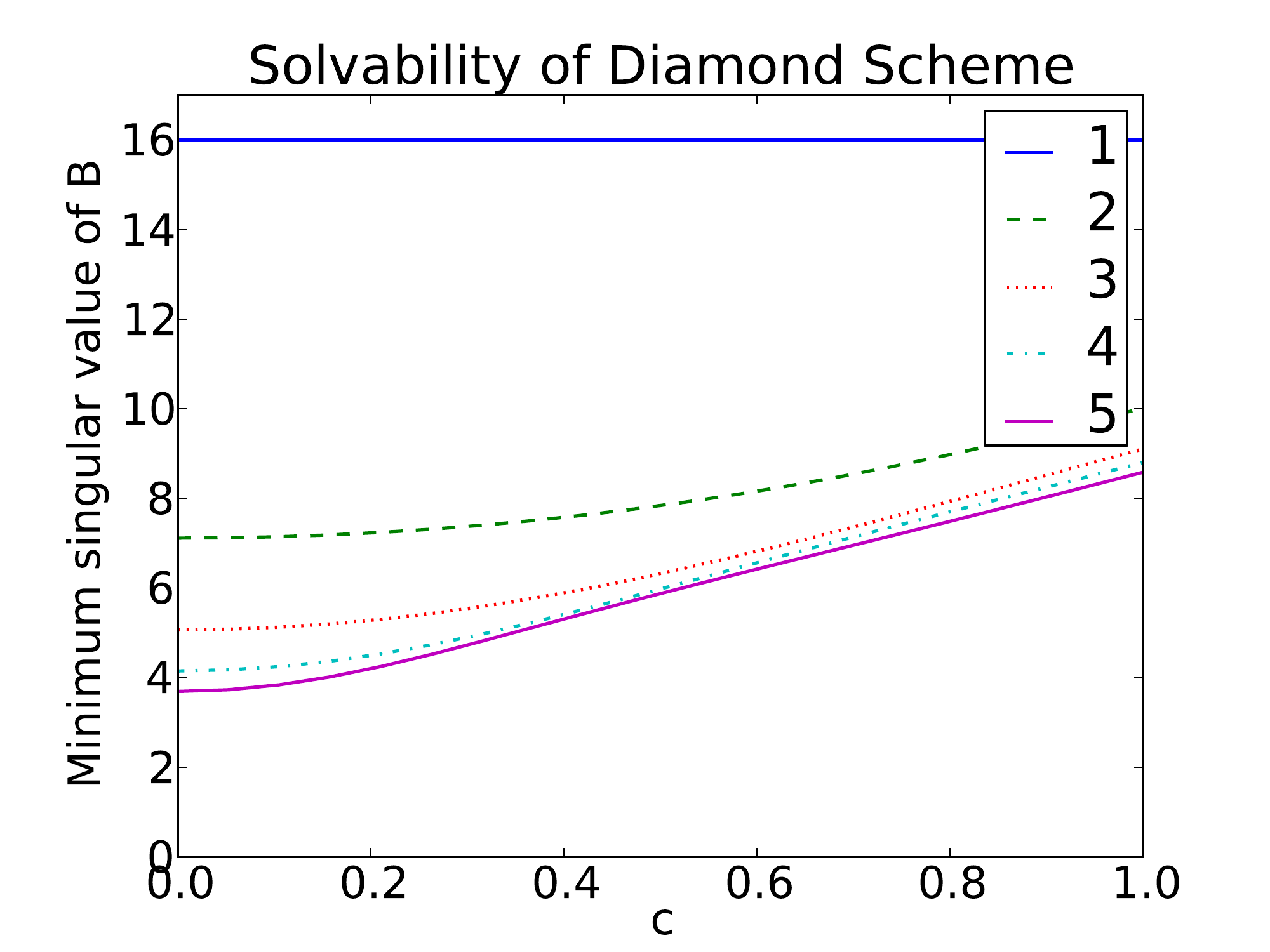}
   \input{\inputdir/min_sv_B_versus_c_caption}
   \label{fig:min_sv_B_versus_c}
\end{figure}

\begin{thm} \label{thm:diamond_conservation_law}
The diamond scheme satisfies the discrete symplectic conservation law
$$
\frac{1}{\Delta t} \sum_{i=1}^r b_i ( \omega_i^\mathrm{t} + \omega_\mathrm{r}^i - (\omega_\ell^i + \omega_i^\mathrm{b}) ) +
\frac{1}{\Delta x} \sum_{i=1}^r b_i ( \kappa_\mathrm{r}^i + \kappa_i^\mathrm{b} - (\kappa_i^\mathrm{t} + \kappa_\ell^i)) ) = 0,
$$
where 
$\omega_n^m = \tfrac{1}{2} d \mathbf{z}_n^m \wedge K d \mathbf{z}_n^m$,
$\kappa_n^m = \tfrac{1}{2} d \mathbf{z}_n^m \wedge L d \mathbf{z}_n^m$,
and $m, n \in \left[0, r\right] \cup \left\{\mathrm{t}, \mathrm{r},
\mathrm{b}, \ell\right\}$ (refer to figure~\eqref{fig:single_diamond}
for the definition of those labels).
\end{thm}
\begin{proof}
The solver within each square satisfies
the discrete multisymplectic conservation law~\cite{ryland}
$$
\Delta x \sum_{i=1}^r b_i ( \tilde{\omega}_i^\mathrm{t} - \tilde{\omega}_i^\mathrm{b} ) +
\Delta t \sum_{i=1}^r b_i ( \tilde{\kappa}_\mathrm{r}^i - \tilde{\kappa}_\ell^i ) = 0,
$$
where $\Delta x = \Delta t = 1$ because the square has side length one. 
Substituting in $\tilde{\omega}_n^m =  \tfrac{1}{2} d \tilde{\mathbf{z}}_n^m \wedge \tilde{K} d \tilde{\mathbf{z}}_n^m$ and
$\tilde{\kappa}_n^m =  \tfrac{1}{2} d \tilde{\mathbf{z}}_n^m \wedge \tilde{L} d \tilde{\mathbf{z}}_n^m$
$$
\frac{1}{2} \sum_{i=1}^r b_i ( d \tilde{\mathbf{z}}_i^\mathrm{t} \wedge \tilde{K} d \tilde{\mathbf{z}}_i^\mathrm{t} - d \tilde{\mathbf{z}}_i^\mathrm{b} \wedge \tilde{K} d \tilde{\mathbf{z}}_i^\mathrm{b} + d \tilde{\mathbf{z}}_\mathrm{r}^i \wedge \tilde{L} d \tilde{\mathbf{z}}_\mathrm{r}^i - d \tilde{\mathbf{z}}_\ell^i \wedge \tilde{L} d \tilde{\mathbf{z}}_\ell^i ) = 0
$$
Using $d \tilde{\mathbf{z}}_n^m = d \mathbf{z}_n^m$, and Eqs.~\eqref{eqn:KLtilde} this becomes
\begin{multline*}
\frac{1}{2} \sum_{i=1}^r b_i ( d \mathbf{z}_i^\mathrm{t} \wedge (\tfrac{1}{\Delta t} K - \tfrac{1}{\Delta x} L) d \mathbf{z}_i^\mathrm{t} - d \mathbf{z}_i^\mathrm{b} \wedge (\tfrac{1}{\Delta t} K - \tfrac{1}{\Delta x} L) d \mathbf{z}_i^\mathrm{b} \\ 
\shoveright{+ d \mathbf{z}_\mathrm{r}^i \wedge (\tfrac{1}{\Delta t} K + \tfrac{1}{\Delta x} L) d \mathbf{z}_\mathrm{r}^i - d \mathbf{z}_\ell^i \wedge (\tfrac{1}{\Delta t} K + \tfrac{1}{\Delta x} L) d \mathbf{z}_\ell^i ) = 0} \\
\Rightarrow \frac{1}{\Delta t} \sum_{i=1}^r b_i ( \omega^\mathrm{t}_i - \omega^\mathrm{b}_i + \omega^i_\mathrm{r} - \omega^i_\ell )    +\frac{1}{\Delta x} \sum_{i=1}^r b_i ( -\kappa^\mathrm{t}_i + \kappa^\mathrm{b}_i + \kappa^i_\mathrm{r} - \kappa^i_\ell ) = 0
\end{multline*}
\end{proof}

We now examine the relationship between the simple diamond scheme
(which uses corner values only) and the $r=1$ diamond scheme (which
uses edge values only). To relate the two, note that the extension to
the corners of the $r=1$ diamond scheme discussed previously, in which
Eqs.~\eqref{eqn:msZ1}, \eqref{eqn:msZ2}, \eqref{eqn:msupdate1},
\eqref{eqn:msupdate2} are applied with $i=j=0$, leads on a single
diamond to
\begin{equation} \label{eqn:equiv1neat}
\begin{aligned}
   \tilde{\mathbf{z}}^\mathrm{b} &= \frac{\mathbf{z}_0^{-1} + \mathbf{z}_1^0}{2}, & \tilde{\mathbf{z}}^\mathrm{t} &= \frac{\mathbf{z}_{-1}^0 + \mathbf{z}_0^1}{2}, \\
   \tilde{\mathbf{z}}_\ell &= \frac{\mathbf{z}_0^{-1} + \mathbf{z}_{-1}^0}{2}, & \tilde{\mathbf{z}}_\mathrm{r} &= \frac{\mathbf{z}_1^0 + \mathbf{z}_0^1}{2},
\end{aligned}
\end{equation}
where the sub/superscript 1 has been dropped.

\begin{thm} \label{thm:r1simpleequiv}
(i) Any solution of the simple diamond scheme, mapped to edge midpoint values
according to Eq.~\eqref{eqn:equiv1neat}, satisfies the equations of the $r=1$ diamond scheme.
(ii) Any solution of the $r=1$ diamond scheme corresponds under
Eq.~\eqref{eqn:equiv1neat} locally to a 1-parameter family of solutions to
the simple diamond scheme. With periodic boundary conditions, the
correspondence is global iff the solution satisfies $\sum_i{\mathbf
z}_{\ell i,j}^1 =\sum_i{\mathbf z}^\mathrm{b}_{1 i,j}$ for all $j$,
where the subscript $_{i,j}$ refers to the $i^\mathrm{th}$ diamond at
the $j^\mathrm{th}$ time step.
\end{thm}
\begin{proof}
When $r=1$ Eqs.~\eqref{eqn:msZ1}--\eqref{eqn:msupdate2} become
\begin{align}
   \mathbf{Z} &= \tilde{\mathbf{z}}_\ell + \frac{1}{2} \mathbf{X} \label{eqn:msZ3r1} \\
   \mathbf{Z} &= \tilde{\mathbf{z}}^\mathrm{b} + \frac{1}{2}  \mathbf{T} \label{eqn:msZ3r2} \\
   \nabla S(\mathbf{Z}_1^1) &= \tilde{K} \mathbf{T} + \tilde{L}  \mathbf{X} \label{eqn:msZ3r3} \\
   \tilde{\mathbf{z}}_\mathrm{r} &= \tilde{\mathbf{z}}_\ell +  \mathbf{X} \label{eqn:msZ3r4} \\
   \tilde{\mathbf{z}}^\mathrm{t} &= \tilde{\mathbf{z}}^\mathrm{b} +  \mathbf{T}, \label{eqn:msZ3r5} 
\end{align}
where $\tilde{K}$ and $\tilde{L}$ are the transformed $K$ and $L$
given in Eq.~\eqref{eqn:KLtilde}, and the sub/superscript 1 has been 
omitted in 
$\tilde{\mathbf{z}}_\ell$,
$\tilde{\mathbf{z}}^\mathrm{b}$,
$\tilde{\mathbf{z}}_\mathrm{r}$,
$\tilde{\mathbf{z}}^\mathrm{t}$, $\mathbf{Z}$, $\mathbf{X}$, and $\mathbf{T}$.

Eliminating $\mathbf{X}$, $\mathbf{T}$, and $\mathbf{Z}$ from the 5 equations
\eqref{eqn:msZ3r1}--\eqref{eqn:msZ3r5} yields the equivalent formulation
\begin{align}
   \tilde{K}\left(\tilde{\mathbf{z}}^\mathrm{t} - \tilde{\mathbf{z}}^\mathrm{b} \right) + \tilde{L} \left(\tilde{\mathbf{z}}_\mathrm{r}-\tilde{\mathbf{z}}_\ell\right) &= \nabla S\left(\frac{\tilde{\mathbf{z}}^\mathrm{t} + \tilde{\mathbf{z}}^\mathrm{b} + \tilde{\mathbf{z}}_\mathrm{r} + \tilde{\mathbf{z}}_\ell}{4}\right),\label{eqn:equivr1neat} \\
   \tilde{\mathbf{z}}^\mathrm{t}- \tilde{\mathbf{z}}_\mathrm{r}+ \tilde{\mathbf{z}}^\mathrm{b}- \tilde{\mathbf{z}}_\ell &= 0. \label{eqn:equivr2neat}
  \end{align}

(i) 
Substituting the relations \eqref{eqn:equiv1neat} in the equations of the simple diamond scheme
\eqref{eqn:discretehampde}, \eqref{eqn:centrept} gives 
\begin{align*}
\nabla S\left(\frac{\mathbf{z}_0^{-1} + \mathbf{z}_{-1}^0 + \mathbf{z}_1^0 + \mathbf{z}_0^1}{4}\right) &=   K \left( \frac{\mathbf{z}_0^1 - \mathbf{z}_0^{-1}}{\Delta t} \right) + L \left( \frac{\mathbf{z}_1^0 - \mathbf{z}_{-1}^0}{\Delta x} \right)   \\
   \Rightarrow
\nabla S\left(\frac{\tilde{\mathbf{z}}^\mathrm{t} + \tilde{\mathbf{z}}^\mathrm{b} + \tilde{\mathbf{z}}_\mathrm{r}+\tilde{\mathbf{z}}_\ell}{4}\right)  &= \tfrac{1}{\Delta t} K \left(\tilde{\mathbf{z}}^\mathrm{t} - \tilde{\mathbf{z}}^\mathrm{b} + \tilde{\mathbf{z}}_\mathrm{r}-\tilde{\mathbf{z}}_\ell \right) + \tfrac{1}{\Delta x} L \left(\tilde{\mathbf{z}}_\mathrm{r}-\tilde{\mathbf{z}}_\ell -\tilde{\mathbf{z}}^\mathrm{t} + \tilde{\mathbf{z}}^\mathrm{b}\right) \\
&= (\tfrac{1}{\Delta t} K -\tfrac{1}{\Delta x} L) \left(\tilde{\mathbf{z}}^\mathrm{t} -
   \tilde{\mathbf{z}}^\mathrm{b} \right) + (\tfrac{1}{\Delta t} K + \tfrac{1}{\Delta x} L)
   \left(\tilde{\mathbf{z}}_\mathrm{r}-\tilde{\mathbf{z}}_\ell\right) && \text{by \eqref{eqn:KLtilde}} \\
&=  \tilde{K}  \left(\tilde{\mathbf{z}}^\mathrm{t} - \tilde{\mathbf{z}}^\mathrm{b} \right) + \tilde{L} \left(\tilde{\mathbf{z}}_\mathrm{r}-\tilde{\mathbf{z}}_\ell\right)  \\
\end{align*}
that is, the equations \eqref{eqn:equivr1neat} of the $r=1$ diamond
scheme are satisfied. Eq.~\eqref{eqn:equivr2neat} follows directly from Eq.~\eqref{eqn:equiv1neat}.

(ii) Using Eq.~\eqref{eqn:equiv1neat}, the corner values of one diamond can be recovered uniquely from the edge values and one corner value.  From these values, adjacent diamonds can be filled in, continuing to get a unique solution for the corner values in any simply-connected region. The same calculation as in part (i) now shows that these corner values satisfy the equations of the simple diamond scheme. For a global solution with periodic boundary conditions, the edge values at one time level $j$ must lie in the range of the mean value operator in \eqref{eqn:equiv1neat}, which gives the condition in the theorem. (If the condition holds at $j=1$, it holds for all $j$, from \eqref{eqn:equivr2neat}.) In both cases one corner value parameterizes the solutions.

\end{proof}

Theorem~\eqref{thm:r1simpleequiv} implies that the multisymplectic conservation laws of the simple and $r=1$ diamond schemes are equivalent under \eqref{eqn:equiv1neat}.  This is now proved directly.
\begin{cor}
Under the relations \eqref{eqn:equiv1neat}, the simple diamond scheme and the $r=1$ diamond scheme have equivalent
discrete multisymplectic conservation laws.
\end{cor}
\begin{proof}
Substitute $r=1$ into Theorem~\eqref{thm:diamond_conservation_law}, note $b_1 = 1$, and differentiate Eq.~\eqref{eqn:equiv1neat} to
get $d \mathbf{z}_1^\mathrm{t} = (d \mathbf{z}_0^1 + d \mathbf{z}_{-1}^0)/2$,
$d \mathbf{z}_\mathrm{r}^1 = (d \mathbf{z}_0^1 + d \mathbf{z}_{1}^0)/2$,
$d \mathbf{z}_1^\mathrm{b} = (d \mathbf{z}_1^0 + d \mathbf{z}_{0}^{-1})/2$, and
$d \mathbf{z}_\ell^1 = (d \mathbf{z}_0^{-1} + d \mathbf{z}_{-1}^0)/2$, leading to
\begin{multline*}
\frac{1}{8\Delta t} \left(  (d \mathbf{z}_0^1 + d \mathbf{z}_{-1}^0) \wedge K  (d \mathbf{z}_0^1 + d \mathbf{z}_{-1}^0) + (d \mathbf{z}_0^1 + d \mathbf{z}_{1}^0) \wedge K (d \mathbf{z}_0^1 + d \mathbf{z}_{1}^0) \right. \\
          \left. - (d \mathbf{z}_0^{-1} + d \mathbf{z}_{-1}^0) \wedge K (d \mathbf{z}_0^{-1} + d \mathbf{z}_{-1}^0) - (d \mathbf{z}_1^0 + d \mathbf{z}_{0}^{-1}) \wedge K (d \mathbf{z}_1^0 + d \mathbf{z}_{0}^{-1})\right) \\
\shoveleft{+ \frac{1}{8\Delta x} \left(  (d \mathbf{z}_0^1 + d \mathbf{z}_{1}^0) \wedge L (d \mathbf{z}_0^1 + d \mathbf{z}_{1}^0) + (d \mathbf{z}_1^0 + d \mathbf{z}_{0}^{-1}) \wedge L (d \mathbf{z}_1^0 + d \mathbf{z}_{0}^{-1}) \right.} \\
          \left. - (d \mathbf{z}_0^1 + d \mathbf{z}_{-1}^0) \wedge L (d \mathbf{z}_0^1 + d \mathbf{z}_{-1}^0) - (d \mathbf{z}_0^{-1} + d \mathbf{z}_{-1}^0) \wedge L (d \mathbf{z}_0^{-1} + d \mathbf{z}_{-1}^0) \right) = 0,
\end{multline*}
which upon expanding and simplifying leads to the  simple diamond scheme conservation law given in proposition~\eqref{prop:sds_conservation_law}.
\end{proof}

\section{Numerical test of diamond scheme}

The diamond scheme with varying $r$ was used to solve the Sine--Gordon
equation as in Section \ref{sec:sinegordon}.  The exact solution is
the so-called \emph{breather} given in Eq.~\eqref{eqn:breather},
and the error is the discrete 2-norm of $u$,
$$
E^2 = \frac{b - a}{N} \sum_i^N \left(\tilde{u}_i - u(a + i \Delta x, T) \right)^2.
$$
The number of diamonds at each time level is $N=40,80,\dots,1280$, and the integration time, $T$, is twice the largest time step. 
The Courant number $\frac{\Delta t}{\Delta x} = \frac{1}{2}$ is held fixed.
The $2rN$ initial values of $z=(u,u_t,u_x)$ needed at the bottom edge of the
first row of diamonds are provided by the exact solution.
The results for the global error are shown in
Fig.~\eqref{fig:error_r12345_sinegordon_periodic_s2_internal0}.  It is apparent that for this problem, the
order
appears to be $r$ when $r$ is odd and $r+1$ when $r$ is even.
\begin{figure}
   \centering
   \includegraphics[width=3in]{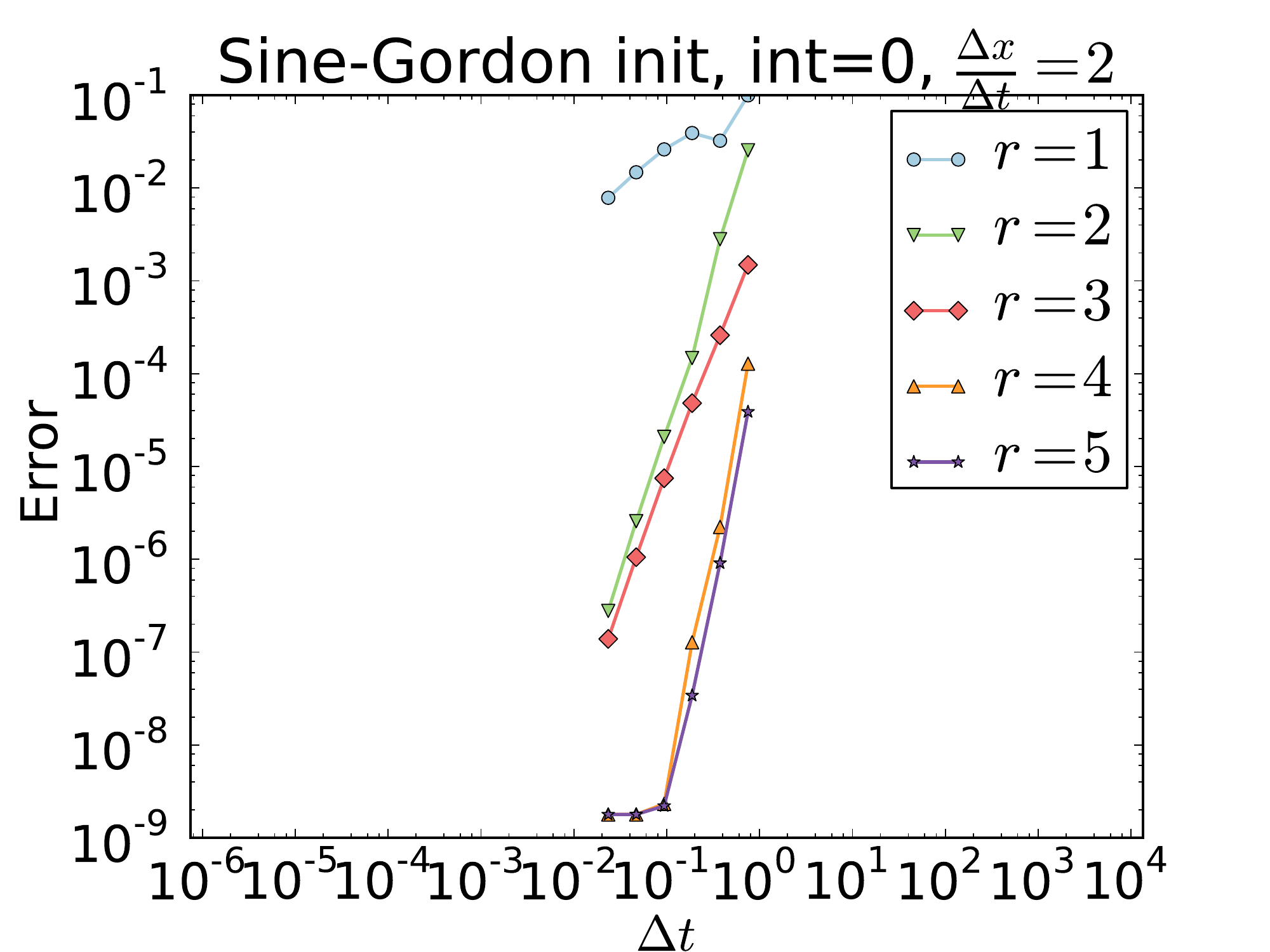}
   \input{\inputdir/error_r12345_sinegordon_periodic_s2_internal0_caption}
   \label{fig:error_r12345_sinegordon_periodic_s2_internal0}
\end{figure}

\section{Dispersion analysis}

\begin{lem} \label{lem:dispersionequation}
For the linear multi-Hamiltonian
\begin{equation}
   K\mathbf{z}_t + L\mathbf{z}_x = S \mathbf{z}, \label{eqn:linearhampde}
\end{equation}
where $S$ is a constant $n \times n$ real symmetric matrix, the
dispersion relation between the wave number $\xi \in \mathbb{R}$ and
frequency $\omega \in \mathbb{R}$  is given by
\begin{equation*}
   p(\xi, \omega) = \mathrm{det}(-i \omega K + i \xi L - S) = 0.
\end{equation*}
\end{lem}
\begin{proof}
Assume 
$
\mathbf{z} =  e^{i (\xi x - \omega t)} \mathbf{c}
$
where $\mathbf{c}$ is a constant vector, is a solution to Eq.~\eqref{eqn:linearhampde}.
Substitution yields
$$
   (-i \omega K + i \xi L - S)  \mathbf{c} = 0.
$$
For non-trivial solutions the matrix on the left must have zero determinant.
\end{proof}

If there are any solutions to $p(\omega,\xi)=0$ with $\xi$ real and $\omega$ complex but not real, then the PDE has solutions that grow without bound.
For example, the dispersion relation for the wave equation, $u_{tt} - u_{xx} =
0$, is $p(\xi, \omega) = \omega(\omega^2-\xi^2)=0$, so all solutions are bounded.  For the
equation $u_{tt} + u_{xx} = 0$, the dispersion relation is $p(\xi,
\omega) = \omega(\omega^2+ \xi^2)=0$, so there are unbounded solutions.

\begin{thm} \label{thm:dispersionsds}
The simple diamond scheme applied to the linear multi-Hamiltonian equation has 
dispersion relation between $ \mathcal{X} \Delta x, \Omega \Delta t \in [-\pi, \pi]$ defined by
\begin{equation*}
   P(\mathcal{X} \Delta x, \Omega \Delta t) = p(h(\mathcal{X} \Delta x, \Omega \Delta t)) = p(h_1(\mathcal{X} \Delta x, \Omega \Delta t), h_2(\mathcal{X} \Delta x, \Omega \Delta t)) = 0,
\end{equation*}
where $p$ is given in lemma~\eqref{lem:dispersionequation} and
\begin{align}
   h(x,y) &= (h_1(x, y), h_2(x, y)), \label{eqn:hdfn} \\
   h_1(x, y) &= \frac{4 \sin \left( \frac{1}{2} x \right)}{\Delta x \left(\cos \left(\frac{1}{2} x \right) +  \cos \left(\frac{1}{2} y \right)\right)}, \nonumber \\
   h_2(x, y) &= \frac{4 \sin \left( \frac{1}{2} y \right)}{\Delta t \left(\cos \left(\frac{1}{2} x \right) +  \cos \left(\frac{1}{2} y \right)\right)}. \nonumber
\end{align}
\end{thm}
\begin{proof}
Assume that a solution to the simple diamond scheme given in Eq.~\eqref{eqn:discretehampde} is
$
\mathbf{z}_j^n =  e^{i (\mathcal{X} j \Delta x - \Omega n \Delta t)} \mathbf{c},
$
where $\mathbf{c}$ is a constant vector, and because adding a multiple of $2 \pi$ to either of $\Omega \Delta t$ or $\mathcal{X} \Delta x$ does not change $\mathbf{z}_j^n$ they are restricted to $[-\pi, \pi]$.
Substitution into Eq.~\eqref{eqn:discretehampde} yields
\begin{equation*}
\begin{split}
   \left[\frac{1}{\Delta t} \left(e^{-i \Omega \frac{1}{2} \Delta t} - e^{i \Omega \frac{1}{2} \Delta t}\right) K +
   \frac{1}{\Delta x} \left(e^{i \mathcal{X} \frac{1}{2} \Delta x} - e^{-i \Omega \frac{1}{2} \Delta x}\right) L - \right. \\
   \left. \frac{1}{4}S\left(e^{-i \Omega \frac{1}{2} \Delta t} + e^{i \Omega \frac{1}{2} \Delta t} + e^{i \mathcal{X} \frac{1}{2} \Delta x} +e^{-i \mathcal{X} \frac{1}{2} \Delta x}   \right) \right] \mathbf{c} &= 0 \\
   \left[\frac{-2i}{\Delta t} \sin\left(\tfrac{1}{2} \Omega \Delta t\right)  K +
         \frac{2i}{\Delta x} \sin\left(\tfrac{1}{2} \mathcal{X} \Delta x\right) L -
         \frac{S}{2} \left(\cos\left(\tfrac{1}{2} \Omega \Delta t\right) + \cos\left(\tfrac{1}{2} \mathcal{X} \Delta x\right) \right) \right] \mathbf{c}  &= 0 \\
   \left[\frac{-4i \sin\left(\tfrac{1}{2} \Omega \Delta t\right)  K }{\Delta t \left(\cos\left(\tfrac{1}{2} \Omega \Delta t\right) + \cos\left(\tfrac{1}{2} \mathcal{X} \Delta x\right) \right) } +
         \frac{4i \sin\left(\tfrac{1}{2} \mathcal{X} \Delta x\right) L }{\Delta x \left(\cos\left(\tfrac{1}{2} \Omega \Delta t\right) + \cos\left(\tfrac{1}{2} \mathcal{X} \Delta x\right) \right) }  -
         S \right] \mathbf{c}  &= 0 \\
   \left[-i h_2(\mathcal{X}\Delta x, \Omega\Delta t) K + i h_1(\mathcal{X}\Delta x, \Omega\Delta t) L - S \right] \mathbf{c}  &= 0
\end{split} 
\end{equation*}
For non-trivial solutions the matrix on the left must have zero determinant, so $p(h_1(\mathcal{X}\Delta x, \Omega\Delta t), h_2(\mathcal{X}\Delta x, \Omega\Delta t)) = 0$.
\end{proof}

\begin{lem}
The $r=1$ diamond scheme applied to the linear
multi-Hamiltonian equation has a 
dispersion relation between $\tilde{\Omega}, \tilde{\mathcal{X}} \in [-\pi, \pi]$ defined by
\begin{equation*}
   \tilde{P}(\tilde{\mathcal{X}}, \tilde{\Omega}) = \mathrm{det}\left( -i 2 \tan\left(\tfrac{\tilde{\Omega}}{2}\right) \tilde{K} + i 2 \tan\left(\tfrac{\tilde{\mathcal{X}}}{2}\right) \tilde{L} - S \right) = 0.
\end{equation*}
The tildes are reminders that this dispersion relation is in the $(\tilde{x}, \tilde{t})$ coordinates.
\end{lem}
\begin{proof}
Assume that a solution to the $r=1$ diamond scheme given in Eqs.~\eqref{eqn:msZ1}--~\eqref{eqn:msupdate2} is
$$
\tilde{\mathbf{z}}_{\tilde{j}}^{\tilde{n}} =  e^{i (\tilde{\mathcal{X}} \tilde{j} - \tilde{\Omega} \tilde{n}) } \mathbf{c},
$$
where $\tilde{\Omega}$ and $\tilde{\mathcal{X}}$ can be restricted to $[-\pi, \pi]$, and
$\mathbf{c}$ is a constant vector.  Note that in $(\tilde{x}, \tilde{t})$ coordinates $\Delta x = \Delta t = 1$.
Substitution into Eqs.~\eqref{eqn:msZ1}--~\eqref{eqn:msupdate2} (or Eq.~\eqref{eqn:equivr1neat} because $r=1$) yields after some simplification
\begin{equation*}
\begin{split}
   \frac{\tilde{K}}{2} \left(e^{-i \tilde{\Omega}} + e^{i (\tilde{\mathcal{X}} - \tilde{\Omega})} - 1 - e^{i \tilde{\mathcal{X}}} \right)\mathbf{c} +
   \frac{\tilde{L}}{2} \left(e^{i \tilde{\mathcal{X}}} + e^{i (\tilde{\mathcal{X}} - \tilde{\Omega})} - 1 - e^{-i \tilde{\Omega}} \right)\mathbf{c} =  \\
   \frac{S}{4} \left(1 + e^{i \tilde{\mathcal{X}}} + e^{-i \tilde{\Omega}} +  e^{i (\tilde{\mathcal{X}} - \tilde{\Omega})} \right)\mathbf{c}.
\end{split} 
\end{equation*}
The result follows after some simplification and using $\tan(x) = \frac{i (1 - e^{2ix})}{1+e^{2ix}}$.
\end{proof}

Recall theorem~\ref{thm:r1simpleequiv}: modulo initial conditions, the $r=1$ diamond scheme and
simple diamond scheme are equivalent.
The following theorem shows that instead of directly calculating the dispersion
relation for the $r=1$ diamond scheme, the dispersion relation from the simple
diamond scheme can simply be transformed from $(x, t)$ coordinates to
$(\tilde{x}, \tilde{t})$ coordinates.

\begin{thm}
The simple and the $r=1$ diamond schemes have identical dispersion relations, that is,  $\tilde{P}(\tilde{\mathcal{X}}, \tilde{\Omega}) = P(\mathcal{X}, \Omega)$.
\end{thm}
\begin{proof}
\begin{align*} 
   \mathbf{z}_j^n &= \tilde{\mathbf{z}}_{\tilde{j}}^{\tilde{n}} \\
\Rightarrow   e^{i (\mathcal{X} j \Delta x - \Omega n \Delta t)} \mathbf{c} &= e^{i (\tilde{\mathcal{X}} \tilde{j} - \tilde{\Omega} \tilde{n}) } \mathbf{c} \\
\Rightarrow   e^{i (\mathcal{X} x - \Omega t)} \mathbf{c} &= e^{i (\tilde{\mathcal{X}} \tilde{x} - \tilde{\Omega} \tilde{t}) } \mathbf{c} \\
\Rightarrow   e^{i (\mathcal{X} \tfrac{\Delta x (\tilde{x} - \tilde{t})}{2} - \Omega \tfrac{\Delta t (\tilde{t} + \tilde{x})}{2} )} \mathbf{c} &= e^{i (\tilde{\mathcal{X}} \tilde{x} - \tilde{\Omega} \tilde{t}) } \mathbf{c}  && \text{using~\eqref{eqn:xttilde}} \\
\Rightarrow   e^{i ( \tfrac{\Delta x \mathcal{X} - \Omega \Delta t}{2} \tilde{x} - \tfrac{\Delta x \mathcal{X} + \Omega \Delta t}{2} \tilde{t})} \mathbf{c} &= e^{i (\tilde{\mathcal{X}} \tilde{x} - \tilde{\Omega} \tilde{t}) } \mathbf{c},
\end{align*} 
thus
\begin{equation} \label{eqn:XOtilde}
   \tilde{\mathcal{X}} = \tfrac{\Delta x \mathcal{X} - \Omega \Delta t}{2} \quad \mathrm{and} \quad \tilde{\Omega} = \tfrac{\Delta x \mathcal{X} + \Omega \Delta t}{2}
\end{equation}
Now
\begin{align*} 
-i &2 \tan\left(\tfrac{\tilde{\Omega}}{2}\right) \tilde{K} + i 2 \tan\left(\tfrac{\tilde{\mathcal{X}}}{2}\right) \tilde{L} - S \\ 
&=  -i 2 \tan\left(\tfrac{\Delta x \mathcal{X} + \Omega \Delta t}{4}\right) \tilde{K} + i 2 \tan\left(\tfrac{\Delta x \mathcal{X} - \Omega \Delta t}{4}\right) \tilde{L} - S && \text{using~\eqref{eqn:XOtilde}}\\
&=  -i 2 \tan\left(\tfrac{\Delta x \mathcal{X} + \Omega \Delta t}{4}\right) \left(\tfrac{1}{\Delta t} K - \tfrac{1}{\Delta x} L\right) + i 2 \tan\left(\tfrac{\Delta x \mathcal{X} - \Omega \Delta t}{4}\right) \left(\tfrac{1}{\Delta t} K + \tfrac{1}{\Delta x} L\right) - S && \text{using~\eqref{eqn:KLtilde}}\\
&=  -i \frac{2}{\Delta t} \left( \tan\left(\tfrac{\Delta x \mathcal{X} + \Omega \Delta t}{4}\right) - \tan\left(\tfrac{\Delta x \mathcal{X} - \Omega \Delta t}{4}\right) \right) K  + \\
& \quad \quad i \frac{2}{\Delta x} \left( \tan\left(\tfrac{\Delta x \mathcal{X} - \Omega \Delta t}{4}\right) + \tan\left(\tfrac{\Delta x \mathcal{X} + \Omega \Delta t}{4}\right) \right) L - S \\
&=  -i \frac{2}{\Delta t} \frac{ 2 \sin\left(\tfrac{\Delta t \Omega}{2}\right) } { \cos\left(\tfrac{\Delta x \mathcal{X}}{2}\right) + \cos\left(\tfrac{\Delta t \Omega}{2}\right)} K + 
                       i \frac{2}{\Delta x} \frac{ 2 \sin\left(\tfrac{\Delta x \mathcal{X}}{2}\right) } { \cos\left(\tfrac{\Delta x \mathcal{X}}{2}\right) + \cos\left(\tfrac{\Delta t \Omega}{2}\right)} L - S
\end{align*} 
where the last step used $\tan(\frac{a+b}{2}) = \frac{\sin(a)+\sin(b)}{\cos(a)+\cos(b)}$.
\end{proof}

\begin{lem} \label{lem:hdiffeo}
Let $U = \left(-\pi, \pi\right) \times
\left(-\pi, \pi\right)$ and let $V=h(U)$, where $h$ is defined in~\eqref{eqn:hdfn}.  The map $h\colon U \rightarrow V$
is a diffeomorphism.
\end{lem}
\begin{proof}
The Jacobian
of $(h_1, h_2)$ is
\begin{equation*}
   J = \frac{2}{\Delta t \left(\cos(\frac{x}{2}) + \cos(\frac{y}{2})\right)^2}
      \begin{pmatrix} \lambda (1 + \cos(\frac{x}{2}) \cos(\frac{y}{2})) & \lambda \sin(\frac{x}{2}) \sin(\frac{y}{2}) \\
                        \sin(\frac{x}{2}) \sin(\frac{y}{2}) & (1 + \cos(\frac{x}{2}) \cos(\frac{y}{2})),
   \end{pmatrix}
\end{equation*}
where $\lambda$ is the Courant number.  By definition $h$ is
surjective, and it is straightforward to show that
$\mathrm{det}(J) \ne 0$ for all $x, y\in U$.  Thus $J$ is a bijection
and $h$ is a local diffeomorphism.  Because both $U$ and $V$
are connected open subsets of $\mathbb{R}^2$, $V$ is simply connected,
and $h$ is proper, $h$ is a diffeomorphism.
\end{proof}

\begin{thm}
   The simple diamond scheme applied to the wave equation is stable
   when
   $\lambda = \frac{\Delta t}{\Delta x} \le 1$.
\end{thm}
\begin{proof}
   It must be shown that for all $\mathcal{X} \Delta x \in
   \left[-\pi, \pi \right]$ there exists $\Omega \Delta t \in
   \left[-\pi,\pi \right]$ such that $P(\mathcal{X} \Delta x,\Omega \Delta t)=0$
   where, from theorem~\ref{thm:dispersionsds}, $P(x,y) = p(h(x,y))=
   p(h_1(x,y),h_2(x,y))$, and for the wave equation $p(\xi, \omega)=\xi^2 -
   \omega^2$.
   
   \begin{align*}
      P(\mathcal{X} \Delta x,\Omega \Delta t) &= 0, \\
      \Leftrightarrow p(h_1(\mathcal{X} \Delta x,\Omega \Delta t), h_2(\Delta x,\Omega \Delta t)) &= 0, \\
      \Leftrightarrow h_1(\mathcal{X} \Delta x,\Omega \Delta t) \pm h_2(\Delta x,\Omega \Delta t) &= 0, \\
      \Leftrightarrow \frac{\sin (\tfrac{1}{2} \Omega \Delta t)}{\sin (\tfrac{1}{2} \mathcal{X} \Delta x)} &= \frac{\Delta t}{\Delta x}, \\
      \Leftrightarrow \Omega \Delta t &= 2 \sin^{-1} \left( \lambda \sin (\tfrac{1}{2} \mathcal{X} \Delta x) \right),
   \end{align*}
   When $\lambda \le 1$ the right hand side can be evaluated for all $\mathcal{X} \Delta x \in
   \left[-\pi, \pi \right]$ and gives $\Omega \Delta t \in
   \left[-\pi,\pi \right]$.
\end{proof}

Another approach will now be illustrated.   The condition that for all $x \in \left[-\pi, \pi \right]$ there exists $y \in \left[-\pi,\pi \right]$ such that $p(h(x,y))=0$
is equivalent to stating that $h: \mathbb{R} \times \left[-\pi, \pi\right]$ contains the solution to $p(\xi,\omega)=0$.  By lemma~\eqref{lem:hdiffeo} $h$ is a diffeomorphism, thus the solution to $p(\xi,\omega)=0$ only 
has to be between the boundaries $h(x,\pm \pi)$.
$$
h(x,\pm \pi) = (h_1(x, \pm \pi), h_2(x, \pm \pi)) = \left(\frac{4 \sin(x/2)}{\Delta x \cos(x/2)}, \pm \frac{4}{\Delta t \cos(x/2)}\right).
$$ 
Let $\xi = h_1(x,y)$ and $\omega = h_2(x,y)$, and use the formula for $\cos \tan^{-1}$ to find
$$
   \omega = \pm \frac{4}{\Delta t} \sqrt{1 + \left(\frac{\Delta x \xi}{4}\right)^2}.
$$
Thus the simple diamond scheme is stable for the wave equation iff
$$
   \xi \le \pm \frac{4}{\Delta t} \sqrt{1 + \left(\frac{\Delta x \xi}{4}\right)^2}.
$$
It is straight forward to check this holds iff $\lambda \le 1$. 
Figure~\eqref{fig:linear_dispersion} illustrates 
the action of $h$ and the linear wave equation dispersion relation.
Figure~\eqref{fig:cubic_dispersion} is similar except for the dispersion relation
$p(\xi, \omega) = \omega - \xi + \xi^3$.
\begin{figure}
   \centering
   \begin{tikzpicture}

   \node[rectangle,minimum size=5.5cm] (top) {\pgfbox[center,center]{\pgfuseimage{dispersion_linear_range_0}}} ;
   \node[rectangle,minimum size=5.5cm] (mid) [below=0.3cm of top] {\pgfbox[center,center]{\pgfuseimage{dispersion_linear_range_1}}};
   \node[rectangle,minimum size=5.5cm] (bot) [below=0.3cm of mid] {\pgfbox[center,center]{\pgfuseimage{dispersion_linear_range_2}}}; 
   \node[rectangle,minimum size=5.5cm] (domain) [left=2cm of mid] {\pgfbox[center,center]{\pgfuseimage{dispersion_linear_domain}}} ;

   \draw[->,gray,ultra thick] (domain) to [out=0,in=180] node[black,above left] {$\frac{\Delta t}{\Delta x} = \frac{2}{1}$}  (top);
   \draw[->,gray,ultra thick] (domain)  to  node[black,above] {$\frac{\Delta t}{\Delta x} = \frac{1}{1}$} (mid);
   \draw[->,gray,ultra thick] (domain)  to [out=0,in=180]  node[black,below left] {$\frac{\Delta t}{\Delta x} = \frac{1}{2}$} (bot);

   \end{tikzpicture}
   \input{\inputdir/linear_dispersion_caption}
   \label{fig:linear_dispersion}
\end{figure}
\begin{figure}
   \centering
   \begin{tikzpicture}

   \node[rectangle,minimum size=5.5cm] (top) {\pgfbox[center,center]{\pgfuseimage{dispersion_cubic_range_0}}} ;
   \node[rectangle,minimum size=5.5cm] (mid) [below=0.3cm of top] {\pgfbox[center,center]{\pgfuseimage{dispersion_cubic_range_1}}};
   \node[rectangle,minimum size=5.5cm] (bot) [below=0.3cm of mid] {\pgfbox[center,center]{\pgfuseimage{dispersion_cubic_range_2}}}; 
   \node[rectangle,minimum size=5.5cm] (domain) [left=2cm of mid] {\pgfbox[center,center]{\pgfuseimage{dispersion_cubic_domain}}} ;

   \draw[->,gray,ultra thick] (domain) to [out=0,in=180] node[black,above left] {$\frac{\Delta t}{\Delta x} = \frac{2}{1}$}  (top);
   \draw[->,gray,ultra thick] (domain)  to  node[black,above] {$\frac{\Delta t}{\Delta x} = \frac{1}{1}$} (mid);
   \draw[->,gray,ultra thick] (domain)  to [out=0,in=180]  node[black,below left] {$\frac{\Delta t}{\Delta x} = \frac{0.025}{1}$} (bot);

   \end{tikzpicture}
   \input{\inputdir/cubic_dispersion_caption}
   \label{fig:cubic_dispersion}

\end{figure}

\section{Discussion}
Many features of the diamond scheme can be seen immediately from its definition:
\begin{enumerate}
\item It is defined for all multi-Hamiltonian systems~\eqref{eqn:hampde}.
\item It is only locally implicit within each diamond. Such locality is suitable for hyperbolic systems with finite wave speeds. Compared to fully implicit schemes like Runge--Kutta box schemes, this leads to
\begin{enumerate}
\item nonlinear equations that have a solution;
\item faster nonlinear solves;
\item better parallelization, as no communication is required during solves and all diamonds are be solved simultaneously--initial experiments indicate
that the scheme scales well with the number of processors; and
\item easier treatment of boundary conditions, which can be handled locally by finding a quadrilateral at the boundary on which the correct amount of information is known.
\end{enumerate}
On the other hand, the implicitness within a diamond should improve stability compared to fully explicit methods in cases where $S(z)$ contributes (a moderate amount of) stiffness to the equation.
\item It is linear in $z$; this is expected to lead to 
\begin{enumerate}
\item better preservation of conservation laws associated with linear symmetries;
\item  better transmission of waves at mesh boundaries; and
\item easier handing of dispersion relations. 
\end{enumerate}
It is the linearity of the method that means it can capture part of the continuous dispersion relation via a remapping of frequencies.
\end{enumerate}

This combination of properties, together with its expected and observed high order, is new for multisymplectic integrators. 

At the same time, the novel mesh introduces some complications:
\begin{enumerate}
\item The implementation is slightly more involved than on a standard mesh; in practice we have not found this to be significant. The parallel implementation is generally easier than on a standard mesh.
\item The interaction of the mesh with the boundaries means that they need special treatment (but at least they {\em can} be treated).
\item The mesh geometry introduces some distortions to the dispersion relation which, in practice, are intermediate in quality between those produced by Runge--Kutta box schemes and those produced by partitioned Runge--Kutta schemes. 
\end{enumerate}

The principle of the diamond method is extremely general and can be applied to a very wide range of PDEs; it and may have applications beyond the multi-symplectic PDE~\eqref{eqn:hampde}. It extends easily to $2d$-hedral meshes for PDEs in $d$-dimensional space-time, again subject to the CFL condition. However, at present to prove existence of solutions to the nonlinear equations we need to restrict to a particular class of equations; ideally one would like to establish existence of numerical solutions for all PDEs~\eqref{eqn:hampde} and relate them to the to the existence of solutions to the PDE itself.

In future work we shall address these issues and establish the order of the diamond method.

\section*{Acknowledgements}
This research was supported by the Marsden Fund of the Royal Society of New Zealand. We would like to thank Reinout Quispel for bringing the staircase method to our attention, and Stephen Marsland for useful discussions.

\bibliographystyle{siam}
\bibliography{\inputdir/diamond}

\end{document}